\documentclass{amsart}

\usepackage[colorlinks=true, urlcolor=black, citecolor=black, linkcolor=black, hyperfootnotes=true]{hyperref}
\usepackage{amssymb}
\usepackage{mathtools}
\usepackage{aliascnt}
\usepackage{mathrsfs}
\usepackage{tikz}
\usepackage{graphicx}

\newcommand*{\C}{\mathrel{\mathsf{C}}}
\newcommand*{\D}{\mathrel{\mathsf{D}}}
\newcommand*{\CC}{\mathrel{\widehat{\mathsf{C}}}}
\newcommand*{\DC}{\mathrel{\widehat{\mathsf{D}}}}

\numberwithin{equation}{section}

\newtheorem{thm}{Theorem}[section]

\newaliascnt{prp}{thm}
\newtheorem{prp}[prp]{Proposition}
\aliascntresetthe{prp}

\newaliascnt{cor}{thm}
\newtheorem{cor}[cor]{Corollary}
\aliascntresetthe{cor}

\newaliascnt{lem}{thm}
\newtheorem{lem}[lem]{Lemma}
\aliascntresetthe{lem}

\theoremstyle{definition}

\newaliascnt{dfn}{thm}
\newtheorem{dfn}[dfn]{Definition}
\aliascntresetthe{dfn}

\newaliascnt{qst}{thm}

\aliascntresetthe{qst}

\newaliascnt{xpl}{thm}
\newtheorem{xpl}[xpl]{Example}
\aliascntresetthe{xpl}

\newaliascnt{rmk}{thm}
\newtheorem{rmk}[rmk]{Remark}
\aliascntresetthe{rmk}

\newtheorem*{rmk*}{\textnormal{\emph{Remark}}}

\author{Tristan Bice}
\email{tristan.bice@gmail.com}
\thanks{The first author is supported by IMPAN (Poland)}
\author{Charles Starling}
\email{cstar@math.carleton.ca}
\thanks{The second author is supported by a Carleton University internal research grant.}

\title{Hausdorff Tight Groupoids Generalised}

\begin{document}

\begin{abstract}
We extend Exel's ample tight groupoid construction to general locally compact \'etale groupoids in the Hausdorff case.  Moreover, we show how inverse semigroups are represented in this way as `pseudobases' of open bisections, thus yielding a duality which encompasses various extensions of the classic Stone duality.
\end{abstract}

\maketitle

\section*{Introduction}

\subsection*{Motivation}

Exel's original tight groupoid construction in \cite{Exel2008} produces an ample (i.e. totally disconnected \'etale) groupoid from an inverse semigroup.  Since then, this general construction has proved extremely successful in producing groupoid models for a vast array of C*-algebras.  The only caveat here is that C*-algebras coming from ample groupoids inevitably have plenty of projections.  Our ultimate goal is to produce combinatorial groupoid models even for projectionless C*-algebras which means we must first find a non-ample generalisation of the tight groupoid construction.

\begin{rmk*}
This task seems all the more urgent given that projectionless C*-algebras like the Jiang-Su and Jacelon-Razak algebras have come to the fore in other contexts, namely the classification program for C*-algebras.  Whether these C*-algebras have any groupoid models at all was only resolved recently in \cite{DeeleyPutnamStrung2015} and \cite{AustinMitra2018}.  The hope would be that an even better picture of these C*-algebras could be provided by more combinatorial groupoid models obtained from a generalised tight groupoid construction.  One might even hope that important C*-algebra properties (e.g. strong self-absorption) could be detected at the inverse semigroup level.
\end{rmk*}

\subsection*{Outline}

In Exel's original tight groupoid construction, the elements of the inverse semigroup get represented as compact open bisections in the resulting groupoid.  To produce more general (potentially non-ample) groupoids we should of course take the inverse semigroup elements to represent more general (potentially non-compact) open bisections.  However, we still require some information about compactness in order to produce locally compact groupoids.

As in domain theory, the key idea is to obtain this information from a (potentially non-reflexive) relation $\prec$ representing `compact containment'.  In compact Hausdorff spaces, this is the same as `closed containment', so in this case
\[p\prec q\qquad\text{represents}\qquad\overline{p}\subseteq q.\]

\begin{rmk*}
We will define several other abstract relations but, as above, in order to provide some intuition, we will say what each relation informally `represents' for subsets of compact Hausdorff spaces.
\end{rmk*}

At first we put the inverse semigroup structure to one side and focus just on the relation $\prec$.  In \autoref{Preliminaries}, we take $\prec$ to be a binary relation on an arbitrary set $P$, from which we define cover relations $\D$ and $\C$.  We then consider centred subsets and finish with a selection principle extending K\"onig's Lemma in \autoref{SelectionPrinciple}.

In \autoref{Pseudobases}, we move on to the topological part of our construction.  With just the single weak interpolation assumption in \eqref{Shrinking}, we are able to show that the tight spectrum always produces a locally compact space in \autoref{TightLocallyCompact}.  Moreover, this yields a duality between abstract and concrete `pseudobases' of locally compact Hausdorff spaces, as we show in \autoref{abstract->concrete} and \autoref{TopologyRecovery}.

In \autoref{InverseSemigroups}, we extend this to inverse semigroups, showing how the the algebraic structure passes to the topological space.  This yields a duality between pseudobasic inverse semigroups and \'etale pseudobases of locally compact Hausdorff \'etale groupoids, as shown in \autoref{LCHEgroupoid} and \autoref{GroupoidRecovery}.

\subsection*{Background}

We have taken inspiration from a number sources, notably domain theory (see \cite{GierzHofmannKeimelLawsonMisloveScott2003} and \cite{Goubault2013}), set theory (see \cite{Kunen2011}), point-free topology (see \cite{PicadoPultr2012}) and its non-commutative generalisations (see \cite{Resende2007}, \cite{LawsonLenz2013} and \cite{KudryavtsevaLawson2017}).  We also discovered that the earlier parts of our theory were developed independently in some of Exel's unpublished notes, namely \cite{Exel2012} and \cite{Exel2012b} (kindly provided to us by the author after presenting our work at the UFSC Operator Algebra Seminar in August 2018), as we mention at relevant points below.

Probably the best way of putting our construction in context is to view it as a generalisation of various Stone type dualities, as summarised in the diagram below.  Note each duality is connected above to all those dualities it generalises and is expressed in the following form:
\vspace{10pt}
\begin{center}
Author (Year) [Paper]\\
Abstract Structure\\
Concrete (Pseudo)basis\\
Topological Space/Groupoid\\
\end{center}
\vspace{10pt}
E.g. the regular ${}^{c\circ}$-$\overline{\cup}^\circ$-$\cap$-bases below considered by De Vries refer to bases of regular open sets (i.e. $O=\overline{O}^\circ$) that are closed under regular complements (i.e. $O^{c\circ}$), regular pairwise unions (i.e. $\overline{O\cup N}^\circ$) and pairwise intersections (i.e. $O\cap N$).  Also LCLH stands for `locally compact locally Hausdorff' (and for compact or Hausdorff spaces we remove the corresponding L).

\begin{rmk*}
Our construction can also be viewed as a generalization of a classical set theoretic construction of a Boolean algebra from a poset (see \cite[Lemma III.4.8]{Kunen2011} and \cite[\S7]{BiceStarling2016}).  Specifically, given any poset $(P,\leq)$, the regular open sets in the Alexandroff topology form a Boolean algebra and, moreover, $P$ is dense in the subalgebra it generates, which can then be identified with its (totally disconnected) Stone space.  The topological part of our construction extends this from posets $(P,\leq\nolinebreak)$ to abstract pseudobases $(P,\prec)$ where $\prec$ need not be reflexive and the resulting spaces are no longer necessarily totally disconnected.
\end{rmk*}

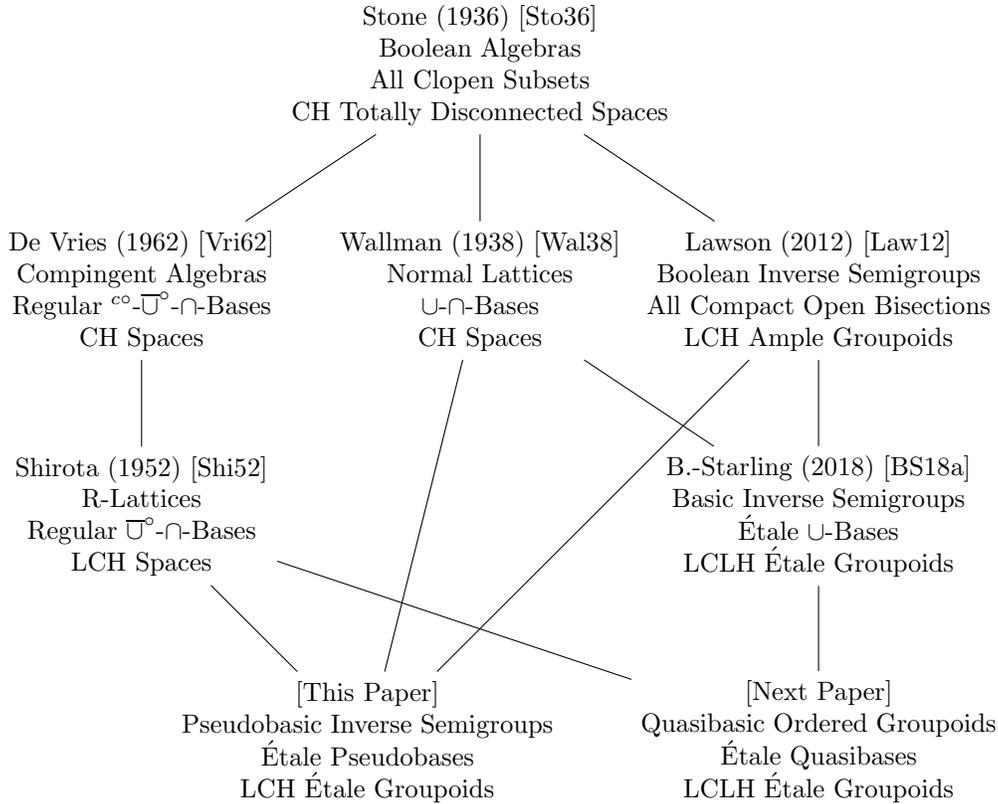
\null
\vspace{10pt}
\hspace{-25pt}\begin{tikzpicture}[scale=3]
    \node (Stone) at (0,0) [align=center]{Stone (1936) \cite{Stone1936}\\ Boolean Algebras\\ All Clopen Subsets\\ CH Totally Disconnected Spaces};
	  \node (De Vries) at (-1.5,-1) [align=center]{De Vries (1962) \cite{DeVries1962}\\ Compingent Algebras\\ Regular ${}^{c\circ}$-$\overline{\cup}^\circ$-$\cap$-Bases\\ CH Spaces};
		\node (Wallman) at (0,-1) [align=center]{Wallman (1938) \cite{Wallman1938}\\ Normal Lattices\\ $\cup$-$\cap$-Bases\\ CH Spaces};
		\node (Lawson) at (1.5,-1) [align=center]{Lawson (2012) \cite{Lawson2012}\\ Boolean Inverse Semigroups\\ All Compact Open Bisections\\ LCH Ample Groupoids};
		\node (Shirota) at (-1.5,-2) [align=center]{Shirota (1952) \cite{Shirota1952}\\ R-Lattices\\ Regular $\overline{\cup}^\circ$-$\cap$-Bases\\ LCH Spaces};
		\node (BiceStarling) at (1.5,-2) [align=center]{B.-Starling (2018) \cite{BiceStarling2018}\\ Basic Inverse Semigroups\\ \'Etale $\cup$-Bases\\ LCLH \'Etale Groupoids};
		\node (BiceStarling2) at (-.5,-3) [align=center]{[This Paper]\\ Pseudobasic Inverse Semigroups\\ \'Etale Pseudobases\\ LCH \'Etale Groupoids};
		\node (BiceStarling3) at (1.5,-3) [align=center]{[Next Paper]\\ Quasibasic Ordered Groupoids\\ \'Etale Quasibases\\ LCLH \'Etale Groupoids};
		
		\draw (Stone)--(De Vries)--(Shirota)--(BiceStarling2);
		\draw (Shirota)--(BiceStarling3);
		\draw (Stone)--(Wallman)--(BiceStarling)--(BiceStarling3);
		\draw (Wallman)--(BiceStarling2);
		\draw (Lawson)--(BiceStarling2);
		\draw (Stone)--(Lawson)--(BiceStarling);
\end{tikzpicture}
\vspace{10pt}

One key difference we should point out is that, while the other dualities above are first order, this is no longer the case with the (generalised) tight groupoid construction.  This is because we must often work with subsets rather than single elements.  However, for the most part we can work with finite subsets so the theory could be expressed in terms of omitting types (as in \cite[\S5]{BiceStarling2016}, which the present paper generalises), which makes it first order in a weak sense.  In particular, the abstract pseudobases we deal with would still provide a natural way of interpreting locally compact spaces in forcing extensions (Wallman's normal lattices were used for this purpose instead in \cite{Kubis2014}).

In contrast, the frames/locales/quantales appearing in point-free topology and their non-commutative generalizations are undeniably second order, making vital use of infinite joins.  Even the finite joins and the corresponding distributivity required in \cite{Lawson2012} and \cite{BiceStarling2018} restricts the freedom one has in making combinatorial constructions.  Thus, while these various point-free counterparts of \'etale groupoids are attractive from an abstract theory point of view, it is rather the (generalised) tight groupoid construction that is more suitable for creating specific examples.

\begin{rmk*}
A similar phenomenon occurs in set theory, where forcing is usually done with combinatorial posets rather than Boolean algebras, even though they produce the same extensions.
\end{rmk*}

\subsection*{Acknowledgements}

We would like to thank Gilles de Castro and Ruy Exel for some very useful comments and questions (see \autoref{CastroQuestion} and \autoref{ExelQuestion}) after presenting our work at the UFSC operator algebra seminar in August 2018.

\section{Preliminaries}\label{Preliminaries}

\begin{center}
\textbf{Throughout we assume $\prec$ is a binary relation on a set $P$}.
\end{center}

First let us make some general notational conventions.
\begin{dfn}
For any $Q\subseteq P$ we define
\[Q^\prec=\{p\in P:\exists q\in Q\ (q\prec p)\}.\]
We extend $\prec$ to a binary relation on $\mathcal{P}(P)=\{Q:Q\subseteq P\}$ by defining
\[Q\prec R\qquad\Leftrightarrow\qquad Q\subseteq R^\succ\qquad\Leftrightarrow\qquad\forall q\in Q\ \exists r\in R\ (q\prec r).\]
\end{dfn}

Note this convention makes $\prec$ `auxiliary' to $\subseteq$ on $\mathcal{P}(P)$, i.e.
\[Q\subseteq Q'\prec R'\subseteq R\qquad\Rightarrow\qquad Q\prec R.\]
This is also the same convention used in \cite[Definition 4.1]{Exel2012} and \cite[\S3]{Lenz2008} (at least for filters, just in reverse).  Also note above that we are using the standard convention of writing $\succ$ for the opposite relation of $\prec$, i.e. $p\succ q$ means $q\prec p$.  For subsets, however, $R\succ Q$ is not the same as $Q\prec R$ \textendash\, the former means $R\subseteq Q^\prec$ while the latter means $Q\subseteq R^\succ$.

\subsection{Cover Relations}

\begin{dfn}\label{Covers}
We define the relations $\D$ and $\C$ on $\mathcal{P}(P)$ by
\begin{align}
\label{DenseCover}\tag{Dense Cover}Q\D R\qquad&\Leftrightarrow\qquad\forall q\prec Q\ \exists r\prec R\ (r\prec q).\\
\label{CompactCover}\tag{Compact Cover}Q\C R\qquad&\Leftrightarrow\qquad\exists\text{ finite }F\ (Q\D F\prec R).
\end{align}
\end{dfn}

Essentially the same relations appear in \cite[Definition 4.2]{Exel2012}.  Note that we again immediately obtain the following auxiliarity properties.
\begin{align}
\label{Daux}Q\subseteq Q'\D R'\subseteq R\qquad&\Rightarrow\qquad Q\D R.\\
Q\subseteq Q'\C R'\subseteq R\qquad&\Rightarrow\qquad Q\C R.
\end{align}
Also note the dense cover relation $\D$ can be expressed more concisely by
\begin{equation}\label{Ddef2}
Q\D R\qquad\Leftrightarrow\qquad Q^\succ\succ R^\succ,
\end{equation}
which just means $Q^\succ\subseteq R^{\succ\prec}$.  Informally, $Q\D R$ is saying that $R$ is dense in $Q$ while $Q\C R$ is saying that the $Q$ is compactly contained in $R$, i.e.
\begin{align*}
Q\D R\qquad&\text{represents}\qquad\bigcup Q\subseteq\overline{\bigcup R}.\\
Q\C R\qquad&\text{represents}\qquad\overline{\bigcup Q}\subseteq\bigcup R.
\end{align*}
(For more precise statements, see \autoref{RDQ}, \autoref{Crep}, \eqref{QDR} and \eqref{QCR}).

Note here it is important that we consider the elements of $P$ to represent \emph{non-empty} sets, i.e. $P$ will \emph{not} have a minimum $0$ (i.e. satisfying $0^\prec=P$) except in trivial cases.  Indeed, in \autoref{InverseSemigroups} when we consider an inverse semigroup $S$ with $0$, we will take $P=S\setminus\{0\}$.  However, $\mathcal{P}(P)$ certainly has a minimum with respect to $\subseteq$, namely the empty set $\emptyset$, which we immediately see is also a minimum with respect to $\prec$, $\D$ and $\C$ on $\mathcal{P}(P)$, i.e. $\emptyset\prec Q$, $\emptyset\D Q$ and $\emptyset\C Q$, for all $Q\subseteq P$.

\begin{rmk}
Alternatively, one could assume that $P$ always has a minimum $0$ and then modify the definitions of $\D$ and $\C$ accordingly.  However, we would then have to say `non-zero' all the time, e.g. `for all non-zero $p$...' or `there exists non-zero $p$...', which can get somewhat tiresome.  
\end{rmk}

\begin{rmk}
In the classical case when $\prec$ is a preorder $\leq$, $\D$ is essentially the same as the Lenz arrow relation in \cite{Lawson2012}, originally from \cite{Lenz2008}.  Equivalently, $Q\D R$ is saying that $R$ is an `outer cover' of $Q^\geq$ in sense of \cite{ExelPardo2016}.  There is also another more topological description of $\D$, which provides extra motivation for calling it the `dense cover relation' (in fact the `dense' terminology is already used for singleton sets in \cite[Definition 11.10]{Exel2008}).  Specifically, $Q\D R$ means that $Q^\geq\cap R^\geq$ is dense in $Q^\geq$ in the Alexandroff topology, where the open sets are precisely the lower sets.  Also, if $R\leq Q$ then, in set theoretic terminology, $Q\D R$ means `$R$ is predense in $Q^\geq$' \textendash\, see \cite[Definition III.3.58]{Kunen2011}.
\end{rmk}

\begin{prp}\label{DCproperties}
If $\prec$ is transitive on $P$, $\D$ and $\C$ are transitive on $\mathcal{P}(P)$ and
\begin{align}
\label{DC=>C}Q\D S\C R\qquad\Rightarrow\qquad Q\C R\qquad&\Rightarrow\qquad Q\D R\qquad\Leftrightarrow\qquad Q\D R^\succ.\\
\label{Dcap}Q\D Q'\quad\text{and}\quad R\D R'\qquad&\Rightarrow\qquad Q^\succ\cap R^\succ\D Q'^\succ\cap R'^\succ.\\
\label{Dcup}Q\D Q'\quad\text{and}\quad R\D R'\qquad&\Rightarrow\qquad Q\cup R\D Q'\cup R'.\\
\label{Ccup}Q\C Q'\quad\text{and}\quad R\C R'\qquad&\Rightarrow\qquad Q\cup R\C Q'\cup R'.
\end{align}
\end{prp}

\begin{proof}
If $Q\prec R\prec S$ then
\[Q\subseteq R^\succ\subseteq S^{\succ\succ}\subseteq S^\succ,\]
the middle $\subseteq$ following from $R\subseteq S^\succ$ and the last $\subseteq$ following from transitivity of $\prec$ on $P$.  Thus $\prec$ is transitive on $\mathcal{P}(P)$ and hence $\D$ is transitive, by \eqref{Ddef2}.

\begin{itemize}
\item[\eqref{DC=>C}] Consequently, $Q\D S\D F\prec R$ implies $Q\D F\prec R$, so $Q\D S\C R$ implies $Q\C R$.  Also $Q\D F\prec R$ implies $Q^\succ\succ F^\succ\subseteq R^{\succ\succ}\subseteq R^\succ$, i.e. $Q\C R$ implies $Q\D R$.  Thus $Q\C R\C S$ implies $Q\D R\C S$ and hence $Q\C S$, i.e. $\C$ is also transitive.  Finally note that if $Q\D R$ then, for all $q\prec Q$, we have $r\prec q\prec Q$ with $r\prec R$.  By transitivity, $r\prec Q$ so we can further take $r'\prec r$ (with $r'\prec R$), so $r'\prec q$ and $r'\in R^{\succ\succ}$, i.e. $Q\D R$ implies $Q\D R^\succ$.

\item[\eqref{Dcap}] Assume $Q^\succ\succ Q'^\succ$ and $R^\succ\succ R'^\succ$.  For any $p\in Q^\succ\cap R^\succ$, we have $p'\prec p$ with $p'\prec Q'$, as $Q^\succ\succ Q'^\succ$.  Then $p'\prec p\prec R$ so $p'\prec R$, by the transitivity of $\prec$, and hence we have $p''\prec p'$ with $p''\prec R'$, as $R^\succ\succ R'^\succ$.  Thus $p''\prec p'\prec Q'$ so $p''\prec Q'$ and $p''\prec p'\prec p$ so $p''\prec p$, again by the transitivity of $\prec$.  Thus $p''\in R'^\succ\cap Q'^\succ$ and, arguing as above, we can further obtain $p'''\prec p''$ so $p'''\in(R'^\succ\cap Q'^\succ)^\succ$.  Then
\[(R^\succ\cap Q^\succ)^\succ\subseteq(R^\succ\cap Q^\succ)\succ(R'^\succ\cap Q'^\succ)^\succ,\]
i.e. $(R^\succ\cap Q^\succ)^\succ\succ(R'^\succ\cap Q'^\succ)^\succ$, as required.

\item[\eqref{Dcup}] Simply note that $Q^\succ\succ Q'^\succ$ and $R^\succ\succ R'^\succ$ implies
\[(Q\cup R)^\succ=Q^\succ\cup R^\succ\succ Q'^\succ\cup R'^\succ=(Q'\cup R')^\succ.\]

\item[\eqref{Ccup}] Now $Q\D F\prec Q'$ and $R\D G\prec R'$ implies $Q\cup R\D F\cup G\prec Q'\cup R'$.\qedhere
\end{itemize}
\end{proof}

\subsection{Round Subsets}

Other properties of $\C$ and $\D$ require an extra assumption.

\begin{dfn}\label{def:round}
We call $R\subseteq P$ \emph{round} if $R\succ R$, i.e.
\[\label{Round}\tag{Round}\forall r\in R\ \exists q\in R\ (q\prec r).\]
\end{dfn}

Equivalently, $R\subseteq P$ is round if iff $R$ has no strictly minimal elements (where $r\in R$ is strictly minimal if $q\not\prec r$, for any $q\in R$).  The `round' terminology is common when talking about ideals, for example, in domain theory \textendash\, see \cite[Proposition 5.1.33]{Goubault2013} or \cite[Proposition III-4.3]{GierzHofmannKeimelLawsonMisloveScott2003}.  As long as $P$ itself is round, $P\D P$ and $\emptyset$ is the only (non-strict) minimum for $\D$ and $\C$.

\begin{prp}
Consider the following statements.
\begin{align}
\label{precMin}Q\prec\emptyset\qquad\Leftrightarrow\qquad Q=\emptyset.\\
\label{DMin}Q\D\emptyset\qquad\Leftrightarrow\qquad Q=\emptyset.\\
\label{CMin}Q\C\emptyset\qquad\Leftrightarrow\qquad Q=\emptyset.
\end{align}
In general, \eqref{precMin} always holds.  If $P$ is round then \eqref{DMin} and \eqref{CMin} also hold.
\end{prp}

\begin{proof}\
\begin{itemize}
\item[\eqref{precMin}]  $Q\prec\emptyset$ means $Q\subseteq\emptyset^\succ=\emptyset$ and hence $Q=\emptyset$.
\item[\eqref{DMin}]  $Q\D\emptyset$ means $Q^\succ\succ\emptyset^\succ=\emptyset$ and hence $Q^\succ=\emptyset$, by \eqref{precMin}.  If $P$ is round then this means $Q=\emptyset$.
\item[\eqref{CMin}]  $Q\C\emptyset$ means $Q\D F\prec\emptyset$ so $F=\emptyset$, by \eqref{precMin}, and hence $Q=\emptyset$, by \eqref{DMin}.\qedhere
\end{itemize}
\end{proof}

\begin{prp}
If $P$ is round and $\prec$ is transitive then
\begin{align}
\label{QsubR}Q\subseteq R\qquad&\Rightarrow\qquad Q\D R.\\
\label{FprecQ}F\prec R\qquad&\Rightarrow\qquad F\C R,\quad\text{for all finite }F\subseteq P.
\end{align}
\end{prp}

\begin{proof}\
\begin{itemize}
\item[\eqref{QsubR}]  If $P$ is round then $Q\subseteq Q^{\succ\prec}$, for all $Q\subseteq P$.  Replacing $Q$ with $R^\succ$ yields $R^\succ\subseteq R^{\succ\succ\prec}\subseteq R^{\succ\prec}$, as $R^\succ\subseteq R^{\succ\succ}$ because $\prec$ is transitive, i.e. $R\D R$.  Thus $Q\subseteq R\D R$ implies $Q\D R$, by \eqref{Daux}.

\item[\eqref{FprecQ}]  By \eqref{QsubR}, $F\prec R$ implies $F\D F\prec R$, which means $F\C R$.\qedhere
\end{itemize}
\end{proof}

\subsection{Centred Subsets}

\begin{dfn}
We define the \emph{formal meet} $\widehat{Q}$ of any $Q\subseteq P$ by
\[\widehat{Q}=\bigcap_{q\in Q}q^\succ\]
We call $Q\subseteq P$ \emph{centred} if $\widehat{F}\neq\emptyset$, for all finite $F\subseteq Q$.
\end{dfn}

The term `centred' comes from \cite[Definition III.3.23]{Kunen2011}.  It will be convenient to also define centred versions of the relations we have considered so far.

\begin{dfn}
For any relation $\sqsubset$ on $\mathcal{P}(P)$, we define a centred version $\widehat{\sqsubset}$ by
\[Q\mathrel{\widehat{\sqsubset}}R\qquad\Leftrightarrow\qquad\exists\text{ finite }F\subseteq Q\ (\widehat{F}\mathrel{\sqsubset}R).\]
\end{dfn}

We are particularly interested in the centred compact cover relation $\CC$, which is related to the centred dense cover relation $\DC$ as before, i.e.
\[Q\CC R\qquad\Leftrightarrow\qquad\exists\text{ finite }F\ (Q\DC F\prec R).\]
We also have the following auxiliarity property (note the first relation is $\supseteq$ not $\subseteq$)
\[Q\supseteq Q'\CC R'\subseteq R\qquad\Rightarrow\qquad Q\CC R\qquad\Rightarrow\qquad Q\C R,\]
at least when $Q\neq\emptyset$ (while $\emptyset\CC R$ means $\widehat{\emptyset}=P\C R$).  Also, by \eqref{CMin},
\[Q\text{ is centred}\qquad\Leftrightarrow\qquad Q\not\CC\emptyset,\]
at least when $P$ is round.  If $Q$ is round then $Q\CC R$ also can be interpreted topologically as saying that the intersection of the sets represented by $Q$ (or their compact closures) are covered by $R$, i.e.
\[Q\CC R\qquad\text{represents}\qquad\bigcap Q\subseteq\bigcup R.\]

\begin{lem}\label{CentredDense}
If $\prec$ is transitive and $F$ is finite then
\[Q\DC F\quad\text{and}\quad Q\not\CC R\qquad\Rightarrow\qquad\exists f\in F\ (Q\cup\{f\}\not\CC R).\]
\end{lem}

\begin{proof}
Assume we have finite $G\subseteq Q$ with $\widehat{G}\D F$ and, for all $f\in F$, we have $Q\cup\{f\}\CC R$, which means we have finite $H_f\subseteq Q$ and $I_f\prec R$ with $\widehat{H}_f\cap\widehat{f}\D I_f$.  We need to show that $Q\CC R$.  Putting $H=\bigcup_{f\in F}H_f$ and $I=\bigcup_{f\in F}I_f$ we then have, for all $f\in F$,
\[\widehat{H}\cap\widehat{f}\ \subseteq\ \widehat{H}_f\cap\widehat{f}\ \D\ I_f\ \subseteq\ I.\]
Now take any $p\in\widehat{G}\cap\widehat{H}$.  As $p\in\widehat{G}\D F$, we have $q\prec p$ and $f\in F$ with $q\prec f$.  Then $q\in\widehat{H}\cup\widehat{f}\D I$ so we have $r\prec q\prec p$ with $r\prec I$.  As $p$ was arbitrary, $\widehat{G}\cap\widehat{H}\D I\prec R$ so $\widehat{G}\cap\widehat{H}\C R$ and hence $Q\CC R$.
\end{proof}

By \eqref{CMin}, we can apply this result with $R=\emptyset$ to obtain the following corollary (note here that $P$ has to be round, otherwise we could potentially have $F=\emptyset$).

\begin{cor}\label{CentredDenseCor}
If $P$ is round, $\prec$ is transitive and $F$ is finite then
\[Q\DC F\text{ and $Q$ is centred}\qquad\Rightarrow\qquad Q\cup\{f\}\text{ is centred, for some }f\in F.\]
\end{cor}

\subsection{Frink Filters}
Recall that  $U\subseteq P$ is a \emph{filter} if
\[\tag{Filter}u,v\in U\qquad\Leftrightarrow\qquad\exists w\in U\ (u,v\succ w).\]
The $\Leftarrow$ part is saying $U$ is upwards closed while the $\Rightarrow$ part is saying $U$ is downwards directed.  In contrast to Exel's construction, our tight subsets need not be filters, at least in the usual sense.  However, they will satisfy the following weaker condition.

\begin{dfn}
We call $U\subseteq P$ a \emph{Frink filter} if
\[\label{FrinkFilter}\tag{Frink Filter}u\in U\qquad\Leftrightarrow\qquad\exists w\prec u\ (U\mathrel{\widehat{\prec}}w).\hspace{50pt}\]
\end{dfn}

The $\Leftarrow$ part of \eqref{FrinkFilter} can be stated more explicitly by saying $u\in U$ whenever $u$ is above some upper bound $w$ for a formal meet $\widehat{F}$, where $F$ is a finite subset of $U$.  In particular, taking singleton $F$, we see that
\[U\text{ is a Frink Filter}\qquad\Rightarrow\qquad U\text{ is upwards closed},\]
as $u\succ w\in U$ implies $\widehat{w}=w^\succ\prec w\prec u$ and hence $u\in U$.

Likewise, the $\Rightarrow$ part of \eqref{FrinkFilter} is immediate if $U$ is round because then, for any $u\in U$, we have $w\prec u$ with $w\in U$ and $\widehat{w}=w^\succ\prec w$.  Conversely, if $P$ satisfies \eqref{Interpolation},
\[U\text{ is a Frink filter}\qquad\Rightarrow\qquad U\text{ is round}.\]
Indeed, $u\in U$ implies $U\mathrel{\widehat{\prec}}w\prec u$, for some $w$, and \eqref{Interpolation} then yields $v$ with $w\prec v\prec u$.  Thus $v\in U$, by the $\Leftarrow$ part of \eqref{FrinkFilter}.

If $P$ is round then, for every $p\in P$, we have $q\in P$ with $(\emptyset\prec)q\prec p$ and hence $P$ is the only possible non-centred Frink filter.  Indeed, as long as $P$ itself is not centred (which will be true except in trivial cases \textendash\, see \autoref{S0tight} \eqref{Ptight} below),
\[U\text{ is a centred Frink filter}\qquad\Leftrightarrow\qquad U\text{ is a proper (i.e. $\neq P$) Frink filter}.\]

\begin{rmk}\label{Frink=>Filter}
For posets, Frink filters are dual to the Frink ideals defined in \cite[\S5 Definition]{Frink1954}.  For $\wedge$-semilattices, Frink filters are just filters.  Indeed, if $(P,\leq)$ is a $\wedge$-semilattice and $F=\{u,v\}\subseteq U$, for some Frink filter $U\subseteq P$, then
\[\widehat{F}=(u\wedge v)^\geq\leq u\wedge v\]
so $u,v\geq u\wedge v\in U$, i.e. $U$ is a filter.  To see why Frink filters might be more natural for general posets, it is instructive to consider Exel's original motivation for introducing tight filters.  First note that when $B$ is a Boolean algebra, the ultrafilters(=maximal proper filters) are closed in $\mathcal{P}(B)$ when we identify each subset of $B$ with its characteristic function in $\{0,1\}^B$ (with its usual product topology).  Exel observed that ultrafilters are no longer closed when $B$ is an arbitrary $\wedge$-semilattice.  To obtain a compact space, one therefore needs to consider the closure of the ultrafilters, i.e. the tight filters.  But when $B$ is an arbitrary poset, even the filters may not be closed and this is why one must consider more general Frink filters.  So as we get more general, this shift can be represented as
\begin{align*}
\text{Ultrafilters}\qquad&\rightarrow\qquad\text{Tight Filters}&&\rightarrow\qquad\text{Tight Frink Filters}.\\
\text{Boolean Algebras}\qquad&\rightarrow\qquad\text{$\wedge$-Semilattices}&&\rightarrow\qquad\text{Posets}.\\
\text{\cite{Stone1936}}\qquad&\rightarrow\qquad\text{\cite{Exel2008}}&&\rightarrow\qquad\text{\cite{BiceStarling2016}}.
\end{align*}
In fact, our tight subsets will automatically be Frink filters \textendash\, see \autoref{Tight=>Frink}.
\end{rmk}

\subsection{Disjoint Subsets}

\begin{dfn}
We define the \emph{disjoint relation} $\perp$ on $\mathcal{P}(P)$ by
\[\label{Disjoint}\tag{Disjoint}Q\perp R\qquad\Leftrightarrow\qquad Q^\succ\cap R^\succ=\emptyset.\]
We also define $Q^\perp=\{p\in P:Q\perp p\}$ so that we always have $Q\perp Q^\perp$ and
\[Q^\perp=\bigcap_{q\in Q}q^\perp.\]
\end{dfn}

As the the name suggests, again thinking of the elements of $P$ as open subsets,
\[Q\perp R\qquad\text{represents}\qquad(\bigcup Q)\cap(\bigcup R)=\emptyset\]
(see \eqref{QperpR} below).  Note $P$ is round if and only if $\perp$ is irreflexive on $P$, i.e.
\[\label{Irreflexive}\tag{Irreflexive}\forall p\ (p\not\perp p).\]
Also $\D$ could have equivalently been defined from $\perp$ as
\[Q\D R\qquad\Leftrightarrow\qquad Q^\succ\cap R^\perp=\emptyset.\]
Moreover, just as $\D$ is auxiliary to $\subseteq$ \textendash\, see \eqref{Daux} \textendash\, $\perp$ is auxiliary to $\D$.

\begin{prp}
If $\prec$ is transitive then
\[\label{perpAuxiliary}\tag{$\perp$-Auxiliary}Q\D Q'\perp R\qquad\Rightarrow\qquad Q\perp R.\]
\end{prp}

\begin{proof}
If we had $Q\not\perp R$, then we would have $q\in Q^\succ\cap R^\succ$.  Then $Q\D Q'$ would yield $q'\prec q$ with $q'\prec Q'$ and hence $q'\in Q'^\succ\cap R^\succ$, i.e. $Q'\not\perp R$.
\end{proof}

\subsection{A Selection Principle}

Later we will need the following selection principle to obtain an certain subset from a family of finite subsets (in fact, for the most part a weaker version which could be derived from K\"{o}nig's lemma would suffice, but the full strength version will be required once in the proof of \autoref{O_FSubset} below).  Roughly speaking,
\begin{enumerate}
\item $\Delta$ represents the finite families of open sets with non-empty intersection.
\item $\Gamma$ represents a round collection of finite families whose unions have the FIP (i.e. finite subsets have non-empty intersection).
\end{enumerate}
The goal is to select one open set in each family in $\Gamma$ in a coherent way so that the chosen sets still form round subset with the finite intersection property.

\begin{lem}\label{SelectionPrinciple}
Say $\Delta$ and $\Gamma$ are collections of finite subsets of $P$ such that
\begin{gather}
\label{succClosed}\tag{$\succ$-Closed}d\succ D\in\Delta\qquad\Rightarrow\qquad\{d\}\cup D\in\Delta\\
\label{precRound}\tag{$\prec$-Round}\forall F\in\Gamma\ \exists G\in\Gamma\ (G\prec F).\\
\label{DeltaCentred}\tag{$\Delta$-Centred}\forall\text{ finite }\Phi\subseteq\Gamma\ \exists D\in\Delta\ \forall F\in\Phi\ (F\cap D\neq\emptyset).
\end{gather}
Then we have round $R\subseteq P$ such that
\begin{gather}
\label{FIP}\tag{$\Delta$-FIP}\forall\text{ finite }D\subseteq R\ (D\in\Delta).\\
\label{Selector}\tag{$\Gamma$-Selector}\forall F\in\Gamma\ (R\cap F\neq\emptyset).
\end{gather}
\end{lem}

\begin{proof}
First we prove some basic facts about $\Gamma$.  To start with, say we had $F\in\Gamma$ and $G'\subseteq G\in\Gamma$ such that
\[F\prec G\setminus G'.\]
We claim that we can remove $G'$ from $G$, i.e. replace $G$ with $G\setminus G'$, without destroying the $\Delta$-centred property of $\Gamma$.  To see this, take any finite $\Phi\subseteq\Gamma$.  As $\Gamma$ is $\Delta$-centred, we have some $D\in\Delta$ such that $D\cap F\neq\emptyset\neq D\cap H$, for all $H\in\Phi$.  As $F\prec G\setminus G'$, we have some $g\in G\setminus G'$ with $g\succ D\cap F$.  Thus $D\cup\{g\}\in\Delta$, by \eqref{succClosed}, and $(D\cup\{g\})\cap(G\setminus G')\supseteq\{g\}\neq\emptyset$ as well.  As $\Phi$ was arbitrary, this proves the claim.

Next, take disjoint subsets $F^+,F^-\subseteq F\in\Gamma$.  We claim that we can remove at least one of these subsets from $F$ again without destroying the $\Delta$-centred property of $\Gamma$.  If not then we could find finite $\Phi^+\subseteq\Gamma$ and $\Phi^-\subseteq\Gamma$ such that $\Phi^+\cup\{F\setminus F^+\}$ and $\Phi^-\cup\{F\setminus F^-\}$ are not $\Delta$-centred.   This means that, for any $D\in\Delta$ such that $D\cap G\neq\emptyset$, for all $G\in\Phi^+$, we have $D\cap F\subseteq F^+$.  Likewise, for any $D\in\Delta$ such that $D\cap G\neq\emptyset$, for all $G\in\Phi^-$, we have $D\cap F\subseteq F^-$.  Thus, for any $D\in\Delta$ such that $D\cap G\neq\emptyset$, for all $G\in\Phi^+\cup\Phi^-$, we have $D\cap F\subseteq F^+\cap F^-=\emptyset$, contradicting the $\Delta$-centred property of $\Gamma$.

Now take another $G\in\Gamma$ with $G\prec F$ and set
\[G^+=\{g\in G:g\not\prec F\setminus F^+\}\qquad\text{and}\qquad G^-=\{g\in G:g\not\prec F\setminus F^-\}.\]
As $G\prec F=(F\setminus F^+)\cup(F\setminus F^-)$, it follows that $G^+$ and $G^-$ are also disjoint.  Also note that $G^+$ and $G^-$ are the smallest subsets of $G$ such that
\[G\setminus G^+\prec F\setminus F^+\qquad\text{and}\qquad G\setminus G^-\prec F\setminus F^-.\]
Thus if we can remove $G_+$ from $G$ without destroying the $\Delta$-centred property of $\Gamma$, we can also remove $F_+$ from $F$, by what we proved above (and likewise for $G^-$).

With these basic properties out of the way, let us now index $\Gamma$ by some set $\Lambda$, i.e. let $\Gamma$ be $(F_\lambda)_{\lambda\in\Lambda}$.  As $\Gamma$ is $\prec$-round, we also have a function $\theta$ on $\Lambda$ such that
\[F_{\theta(\lambda)}\prec F_\lambda,\]
for all $\lambda\in\Lambda$.  Now consider the collection $\Theta_\Lambda$ of all $\Delta$-centred $\Lambda$-indexed families $(G_\lambda)_{\lambda\in\Lambda}$ of finite subsets of $P$ such that $G_{\theta(\lambda)}\prec G_\lambda$, for all $\lambda\in\Lambda$.  Order $\Theta_\Lambda$ by
\[(G_\lambda)_{\lambda\in\Lambda}\leq(H_\lambda)_{\lambda\in\Lambda}\qquad\Leftrightarrow\qquad\forall\lambda\in\Lambda\ (G_\lambda\subseteq H_\lambda).\]

We claim that any decreasing (potentially transfinite) sequence in $\Theta_\lambda$
\[(G^1_\lambda)_{\lambda\in\Lambda}\geq(G^2_\lambda)_{\lambda\in\Lambda}\geq\ldots\]
has a lower bound in $\Theta_\Lambda$ defined by $G_\lambda=\bigcap_\alpha G^\alpha_\lambda$.  To see this, note that for each $\lambda$, $G_\lambda$ and $G_{\theta(\lambda)}$ are finite so there must be some $\alpha$ with $G_\lambda=G^\alpha_\lambda$ and $G_{\theta(\lambda)}=G^\alpha_{\theta(\lambda)}$.  As $(G^\alpha_\lambda)_{\lambda\in\Lambda}$ is in $\Theta_\Lambda$, it follows that $G_{\theta(\lambda)}\prec G_\lambda$.  Likewise, for any finite $\Phi\subseteq\Lambda$, we must have $\alpha$ such that $G_\lambda=G^\alpha_\lambda$, for all $\lambda\in\Phi$.  As $(G^\alpha_\lambda)_{\lambda\in\Lambda}$ is $\Delta$-centred, it follows that we can find $G\in\Delta$ such that $G_\lambda\cap G\neq\emptyset$, for all $\lambda\in\Phi$.  This shows that $(G_\lambda)_{\lambda\in\Lambda}$ is also $\Delta$-centred and hence in $\Theta_\Lambda$.

Thus, by the Kuratowski-Zorn lemma, we have some minimal $(G_\lambda)_{\lambda\in\Lambda}$ below the given $(F_\lambda)_{\lambda\in\Lambda}$.  As $(G_\lambda)_{\lambda\in\Lambda}$ is $\Delta$-centred, each $G_\lambda$ is non-empty.  In fact, we claim that each $G_\lambda$ must be a singleton.  If not then some $G_\gamma$ has disjoint non-empty subsets $G^+_\gamma$ and $G^-_\gamma$.  Let $\theta^0(\gamma)=\gamma$ and $\theta^{n+1}(\gamma)=\theta(\theta^n(\gamma))$ and, for $*=+$ or $-$, recursively define
\[G^*_{\theta^{n+1}(\gamma)}=\{g\in G_{\theta^{n+1}(\gamma)}:g\not\prec G_{\theta^n(\gamma)}\setminus G^*_{\theta^n(\gamma)}\}.\]
As noted above, $G^+_{\theta^n(\gamma)}$ and $G^-_{\theta^n(\gamma)}$ are disjoint and
\begin{equation}\label{Removals}
(G_{\theta^{n+1}(\gamma)}\setminus G^*_{\theta^{n+1}(\gamma)})\prec(G_{\theta^n(\gamma)}\setminus G^*_{\theta^n(\gamma)}),
\end{equation}
so if we can remove $G^*_{\theta^{n+1}(\gamma)}$ from $G_{\theta^{n+1}(\gamma)}$ then we can also remove $G^*_{\theta^n(\gamma)}$ from $G_{\theta^n(\gamma)}$, without destroying the $\Delta$-centred property of $\Gamma$.  But for each $n$, we must be able to remove $G^*_{\theta^n(\gamma)}$ from $G_{\theta^n(\gamma)}$ for either $*=+$ or $-$.  Thus for either $*=+$ or $-$, we can in fact remove $G^*_{\theta^n(\gamma)}$ from $G_{\theta^n(\gamma)}$ for all $n$.  Then the resulting $(G'_\lambda)_{\lambda\in\Lambda}$ will not only remain $\Delta$-centred but also still satisfy $G'_{\theta(\lambda)}\prec G'_\theta$, by \eqref{Removals}.  But this means $(G'_\lambda)_{\lambda\in\Lambda}$ is in $\Theta_\Lambda$, even though it is strictly below $(G_\lambda)$, contradicting the minimality of $(G_\lambda)_{\lambda\in\Lambda}$.  So each $G_\lambda$ must have been a singleton, proving the claim.

Finally, let $R=\bigcup_\lambda G_\lambda$, which certainly satisfies \eqref{Selector}.  Also $R$ is round, as $(G_\lambda)_{\lambda\in\Lambda}$ consists of $\prec$-round singletons.  Moreover, $R$ satisfies \eqref{FIP}, as $(G_\lambda)_{\lambda\in\Lambda}$ consists of $\Delta$-centred singletons.
\end{proof}

\section{Pseudobases}\label{Pseudobases}

\subsection{The Shrinking Condition}

From now on we will focus on transitive relations $\prec$ that make $P$ round.  However, there is one more condition that we need to be able to define general locally compact Hausdorff spaces from $(P,\prec)$.  Specifically, we need a condition which represents the ability to `shrink' a cover to another cover, each element of which is compactly contained in some element of the original cover.

\begin{dfn}
$(P,\prec)$ is an \emph{abstract pseudobasis} if $P$ is round, $\prec$ is transitive and
\[\label{Shrinking}\tag{Shrinking}p\prec q\qquad\Rightarrow\qquad p\C q^\succ.\]
\end{dfn}

More explicitly, this condition is saying
\[\tag{Shrinking}p\prec q\qquad\Rightarrow\qquad\exists\text{ finite }F\ (p\C F\prec q).\]
This makes it clear that \eqref{Shrinking} is a kind of weak interpolation condition.

\begin{rmk}\label{InterpolationRemark}
By \eqref{FprecQ}, \eqref{Shrinking} weakens the more standard interpolation condition considered in domain theory (see \cite[Definition I-1.17]{GierzHofmannKeimelLawsonMisloveScott2003}) given by
\[\label{Interpolation}\tag{Interpolation}p\prec q\qquad\Rightarrow\qquad\exists r\ (p\prec r\prec q).\]
In \cite[Definition 2.1 (v)]{Exel2012} and \cite[Definition 1.2 (v)]{Exel2012b}, \eqref{Interpolation} is also considered as a kind of regularity/normality condition.  In particular, any poset becomes an abstract pseudobasis when we take $\prec\ =\ \leq$, as reflexive relations are trivially round and satisfy \eqref{Interpolation}.

However, it would seem that \eqref{Shrinking} is more appropriate for our work than \eqref{Interpolation}, as \eqref{Shrinking} applies to arbitrary (pseudo)bases of locally compact Hausdorff spaces, even when they are not closed under taking finite unions.  The extra freedom coming from \eqref{Shrinking} will also be convenient for constructing specific examples.  In any case, \eqref{Shrinking} suffices to guarantee that we have a large supply of tight subsets, as we will soon see.
\end{rmk}

First we note \eqref{Shrinking} extends to $\C$, allowing us to replace $\D$ with $\C$ in \eqref{DC=>C}.

\begin{prp}
If $(P,\prec)$ is an abstract pseudobasis then
\begin{equation}\label{RCQ}\tag{$\mathsf{C}$-Shrinking}
R\C Q\qquad\Rightarrow\qquad R\C Q^\succ.
\end{equation}
\end{prp}

\begin{proof}
First note that, for all finite $F\subseteq P$, \eqref{Shrinking} and \eqref{Ccup} yields
\[F\prec Q\qquad\Rightarrow\qquad F\C Q^\succ.\]
Thus if $R\C Q$, i.e. $R\D F\prec Q$, for some finite $F\subseteq P$, then $R\D F\C Q^\succ$ and hence $R\C Q^\succ$, by \eqref{DC=>C}.
\end{proof}

\subsection{Tight Subsets}

\begin{center}
\textbf{From now on we assume $(P,\prec)$ is an abstract pseudobasis.}
\end{center}

Again recall that we are thinking of elements of $P$ as open subsets of a topological space.  It should then be possible to recover the points of the space from their open neighbourhoods.  The key thing to note here is that no neighborhood of a point $x$ can be covered by subsets which do not contain $x$.  Moreover, in locally compact spaces every neighborhood of $x$ compactly contains a smaller neighbourhood of $x$.  Thus the neighbourhoods of $x$ in $P$ are `tight' in the following sense.

\begin{dfn}\label{TightFilter}
We call $T\subseteq P$ \emph{tight} if $T\succ T\not\CC P\setminus T$.
\end{dfn}

So $T$ is tight iff $T$ is round and, for all finite $F,G\subseteq P$,
\begin{equation}\label{Tight-Round}
F\subseteq T\quad\text{and}\quad G\prec P\setminus T\qquad\Rightarrow\qquad\widehat{F}\cap G^\perp\neq\emptyset
\end{equation}
(because $\widehat{F}\cap G^\perp=\emptyset$ would imply $\widehat{F}\D G\prec P\setminus T$ and hence $\widehat{F}\C P\setminus T$).  

\begin{prp}\label{Tight=>Frink}
If $T\subseteq P$ is tight then $T$ is a Frink filter.
\end{prp}

\begin{proof}
The $\Rightarrow$ part of \eqref{FrinkFilter} is immediate because $T$ is round.  Conversely, if $T\mathrel{\widehat{\prec}}v\prec u$ then $T\mathrel{\widehat{\D}}v\prec u$ so $T\mathrel{\widehat{\C}}u$.  Thus $u\in T$, as $T\not\mathrel{\widehat{\C}}P\setminus T$.
\end{proof}

In particular, any tight $T$ is upwards closed and thus $T=T^\prec$, as $T$ is also round.

\begin{rmk}
If $(P,\prec)=(S\setminus\{0\},\leq)$, for some $\wedge$-semilattice $S$ with minimum $0$, then every tight subset is a filter, by \autoref{Frink=>Filter} and \autoref{Tight=>Frink}.  In this case, our tight subsets are precisely the tight filters defined in \cite{ExelPardo2016}.  This becomes clear from the following general characterisation of tight subsets, as \eqref{Tight'} below extends \cite[(2.10)]{ExelPardo2016}.
\end{rmk}

\begin{prp}\label{tightChars}
If $T\neq P$ then $T$ is tight iff either of the following holds.
\begin{align}
\label{Tight}\tag{Tight}t\in T\quad&\Leftrightarrow\quad\exists\text{ finite }F\subseteq T\ \exists\text{ finite }G\prec P\setminus T\ (\widehat{F}\cap G^\perp\C t).\\
\label{Tight'}\tag{Tight$'$}\emptyset\neq H\cap T\quad&\Leftrightarrow\quad\exists\text{ finite }F\subseteq T\ \exists\text{ finite }G\prec P\setminus T\ (\widehat{F}\cap G^\perp\C H).
\end{align}
\end{prp}

\begin{proof}
If $T$ is round then, for any $t\in T$, we have $s\in T$ with $s\prec t$ and hence $\widehat{s}\cap\emptyset^\perp\C t$.  If $T\not\CC P\setminus T$ and we have finite $F\subseteq T$ and $G\prec P\setminus T$ with $\widehat{F}\cap G^\perp\C H$ then we have finite $I$ with $\widehat{F}\cap G^\perp\D I\prec H$ and hence $\widehat{F}\D G\cup I$.  If we had $H\cap T=\emptyset$ then this would imply $\widehat{F}\D G\cup I\prec P\setminus T$ and hence $T\CC P\setminus T$, a contradiction, so $H\cap T\neq\emptyset$.  Thus if $T$ is tight then $T$ satisfies \eqref{Tight'}.  Also \eqref{Tight'} implies \eqref{Tight}, as the latter is just the former restricted to singleton $H$.

Now say $T$ satisfies \eqref{Tight}.  For any $t\in T$, we have finite $F\subseteq T$ and $G\prec P\setminus T$ with $\widehat{F}\cap G^\perp\C t$.  By \eqref{RCQ}, $\widehat{F}\cap G^\perp\C t^\succ$ so we have finite $H,I\subseteq P$ with $\widehat{F}\cap G^\perp\D I\prec H\prec t$ and hence $\widehat{F}\cap(G\cup I)^\perp=\emptyset\C p$, for all $p\in P$.  Then $H\subseteq P\setminus T$ would imply $G\cup I\prec P\setminus T$ and hence $T=P$, by \eqref{Tight}, contradicting our assumption.  Thus $H\cap T\neq\emptyset$, which shows that $T$ is round.

Still assuming $T$ satisfies \eqref{Tight}, say we had $T\CC P\setminus T$, which means we have finite $F\subseteq T$ and $G\prec P\setminus T$ with $\widehat{F}\D G$.  Then $\widehat{F}\cap G^\perp=\emptyset\C p$, for all $p\in P$, so $T=P$, by \eqref{Tight}, again contradicting our assumption.  Thus $T\not\CC P\setminus T$, showing that \eqref{Tight} does indeed imply that $T$ is tight.
\end{proof}

Incidentally, $P$ itself will only be tight in trivial cases.  At the other extreme, it is more common for $\emptyset$ to be tight, although we will specifically omit $\emptyset$ from the tight spectrum considered in the next section (in order to be able to obtain locally compact rather than just compact spaces).

\begin{prp}\label{S0tight}\
\begin{enumerate}
\item\label{Ptight} $P$ is tight iff $P$ is centred, in which case $P$ is the only tight subset of $P$.
\item\label{0tight} $\emptyset$ is tight iff $P\not\C P$.
\end{enumerate}
\end{prp}

\begin{proof}\
\begin{itemize}
\item[\eqref{Ptight}]  By definition, $P$ is tight iff $P\not\CC\emptyset$.  This means $\widehat{F}\not\C\emptyset$, for all finite $F\subseteq P$, which is just saying that $P$ is centred.  In particular, this implies $P\D p$ and hence $P\C p$, for any $p\in P$.  Thus $T\CC P\setminus T$ whenever $T\neq P$ so there can be no other tight subsets.

\item[\eqref{0tight}]  By definition, $\emptyset$ is tight iff $\emptyset\not\CC P$, which means $\widehat{\emptyset}=P\not\C P$. \qedhere
\end{itemize}
\end{proof}

Thus, instead of $T\neq P$, we could assume $T$ is centred in \autoref{tightChars}, i.e.
\[T\text{ is tight}\qquad\Leftrightarrow\qquad T\text{ is centred and satisfies \eqref{Tight}}.\]

While round centred subsets need not be tight (or even upwards closed), they always have tight extensions.  Indeed, by the Kuratowski-Zorn Lemma, any such subset is contained in one that is maximal among all round centred subsets.  By the following result, these subsets will necessarily be tight and, in particular, Frink filters.  Thus this result generalizes \cite[Proposition 12.7]{Exel2008}, which says that ultrafilters in $\wedge$-semilattices are necessarily tight.

\begin{prp}\label{maximalroundcentredtight}
If $U$ is round and centred then
\begin{align}
\label{UltraExists}U\text{ is maximal}\qquad&\Leftrightarrow\hspace{53pt}\forall a\prec P\setminus U\ (U\mathrel{\widehat{\perp}}a)\\
\label{UltraForall}&\Leftrightarrow\qquad\forall\text{ finite }G\prec P\setminus U\ (U\mathrel{\widehat{\perp}}G)\\
\label{TightForall}\Rightarrow\quad U\text{ is tight}\qquad&\Leftrightarrow\qquad\forall\text{ finite }G\prec P\setminus U\ (U\not\DC G).
\end{align}
\end{prp}

\begin{proof}\
\begin{itemize}
\item[\eqref{UltraExists}]  Any round extension of $U$ must contain some $a\prec P\setminus U$.  So if the right side of \eqref{UltraExists} holds then such an extension could not be centred, which implies that $U$ must be maximal.  Conversely, say the right side of \eqref{UltraExists} fails, so we have some $a\prec P\setminus U$ with $\widehat{F}\not\perp a$, for all finite $F\subseteq U$, i.e. $U\cup\{a\}$ is centred.  By \eqref{Shrinking}, we have a sequence $(F_n)$ of finite subsets of $P$ such that, for all $n$,
\[a\C F_{n+1}\prec F_n\prec P\setminus U.\]
In particular, for all $n$, $a\D F_n$ which implies that $U\cup\{f\}$ is centred, for some $f\in F_n$, by \autoref{CentredDenseCor} (with $Q=U\cup\{a\}$ and $F=F_n$).  Thus we can apply \autoref{SelectionPrinciple}, taking $(F_n)$ for $\Gamma$ and setting
\[\Delta=\{D\subseteq P:D\text{ is finite and $U\cup D$ is centred}\},\]
to obtain round $T$ such that $U\cup T$ is centred and $T\cap F_n\neq\emptyset$, for all $n$ (actually $F'_n=\{f\in F_n:U\cup\{f\}\text{ is centred}\}$ defines a finitely branching $\omega$-tree so in this case we could also obtain $T$ from a simple application of K\"onig's lemma \textendash\, see \cite[Lemma III.5.6]{Kunen2011}).  Thus $U\cup T^\prec$ is a round centred subset containing some element of $P\setminus U$ so $U$ could not have been maximal.

\item[\eqref{UltraForall}]  We immediately see that \eqref{UltraForall} implies \eqref{UltraExists}.  Conversely, if \eqref{UltraExists} holds and we are given finite $G\prec P\setminus U$ then, for each $g\in G$, we have finite $F_g\subseteq U$ with $\widehat{F}_g\perp g$.  Taking $F=\bigcup_{g\in G}F_g$, we then have $\widehat{F}\perp G$.

\item[\eqref{UltraForall}$\Rightarrow$\eqref{TightForall}]  If \eqref{UltraForall} holds and we are given $G\prec P\setminus U$ then we have finite $F\subseteq U$ with $\widehat{F}\perp G$.  For any other finite $F'\subseteq U$, we have $\emptyset\neq\widehat{F\cup F'}\perp G$, as $U$ is centred, so $\widehat{F}'\not\D G$.

\item[\eqref{TightForall}]  Just note that the right side of \eqref{TightForall} means $U\not\CC P\setminus U$.\qedhere
\end{itemize}
\end{proof}

The following result shows that we can even be more selective about our tight extensions.  For example, we might want to find a tight extension of some round $R$ but still avoid another given subset $S$.  The following result (in the $Q=\emptyset$ case) shows that this can be done as long as $R\not\CC S$.

\begin{rmk}
Consequently, this result is essentially an extension of Birkhoff's prime ideal theorem for distributive lattices (see \cite[Lemma I-3.20]{GierzHofmannKeimelLawsonMisloveScott2003}).  Indeed, if $(P,\prec)=(L\setminus\{0\},\leq)$, for some separative distributive lattice $L$ with minimum $0$, then the tight subsets are precisely the prime filters.  In this case, $R\not\CC S$ is saying that the filter generated by $R$ is disjoint from the ideal generated by $S$.  Birkhoff's theorem says that $R$ therefore extends to a prime filter disjoint from $S$.
\end{rmk}

\begin{thm}\label{TightStretching}
For any $Q,R,S\subseteq P$,
\[R\text{ is round and }Q\cup R\not\CC S\qquad\Rightarrow\qquad\exists\text{ tight }T\supseteq R\ (Q\cup T\not\CC P\setminus T\supseteq S).\]
\end{thm}

\begin{proof}
As $Q\cup T\not\CC S$ depends only on the finite subsets of $T$, we can apply the Kuratowski-Zorn lemma to obtain maximal round $T\supseteq R$ with $Q\cup T\not\CC S$.  Looking for a contradiction, assume that $Q\cup T\CC P\setminus T$, which means we have finite $F\subseteq Q\cup T$ and $G\prec P\setminus T$ with $\widehat{F}\D G$.  By \eqref{FprecQ} and \eqref{RCQ}, we have a sequence $(G_n)$ of finite subsets of $P$ such that
\[\widehat{F}\D G\C G_{n+1}\prec G_n\prec P\setminus U,\]
and hence $\widehat{F}\D G_n$, for all $n$.  Thus for each $n$, we have $g\in G_n$ with $Q\cup T\cup\{g\}\not\CC S$, by \autoref{CentredDense}.  Thus we can apply \autoref{SelectionPrinciple}, taking $(G_n)$ for $\Gamma$ and setting
\[\Delta=\{D\subseteq P:D\text{ is finite and }Q\cup T\cup D\not\CC S\}\]
to obtain round $U$ such that $Q\cup T\cup U\not\CC S$ (actually, again K\"onig's lemma would suffice here).  But then $U^\prec\cap P\setminus T\neq\emptyset$ so $T\cup U^\prec$ would be a proper round extension of $T$ with $Q\cup T\cup U^\prec\not\CC S$, contradicting maximality.  Thus we must have indeed had $Q\cup T\not\CC P\setminus T$.  In particular, $T\not\CC P\setminus T$ so $T$ is tight, as $T$ is also round.  Lastly note that if we had $a\in T\cap S$ then, as $T$ is round, we would have $b\prec a$ with $b\in T\subseteq Q\cup T\not\CC S\ni a$ and hence $b\not\C a$, contradicting \eqref{FprecQ}, so $S\subseteq P\setminus T$.
\end{proof}

\subsection{The Tight Spectrum}

\begin{center}
\textbf{Recall our standing assumption that $P$ is an abstract pseudobasis.}
\end{center}

\begin{dfn}
The \emph{tight spectrum} $\mathcal{T}(P)$ of $P$ is the set of non-empty tight subsets of $P$ with the topology given by the following basis, for finite $F,G\subseteq P$.
\[O_F^G=\{T\in\mathcal{T}(P):F\subseteq T\text{ and }G\C P\setminus T\}.\]
\end{dfn}

\begin{rmk}
If we allowed $\emptyset$ to be part of the tight spectrum then we would always obtain compact rather than locally compact spaces.  A more serious issue arises for inverse semigroups, where we need to avoid the empty subset to ensure that multiplying tight subsets yields a groupoid operation.
\end{rmk}

For convenience, we often omit empty sets and write $O_F=O_F^\emptyset$ and $O^G=O_\emptyset^G$.  Incidentally, note that $O_\emptyset^\emptyset=\mathcal{T}(P)$, which of course would be open even if we did not include it in the basis.  Also note that \eqref{Ccup} yields
\begin{equation}\label{OFGdef}
O_F^G=\bigcap_{f\in F}O_f\cap\bigcap_{g\in G}O^g
\end{equation}
so the given sets are closed under finite intersections and really do form a basis.

\begin{rmk}\label{CastroQuestion}
For anyone familiar with Stone duality or point-free topology, it might be tempting to consider the weaker topology on tight subsets generated just by the sets $O_f=\{T\in\mathcal{T}(P):f\in T\}$.  The problem is that, for general abstract pseudobases, the resulting space may not be Hausdorff or even $T_1$, although it will still be locally compact and $T_0$.  A standard way of strengthening such a topology to make it Hausdorff is to consider the patch topology as in \autoref{PseudoPatch} below.  Indeed, combined with \autoref{abstract->concrete}, this shows that the resulting patch topology is precisely the topology obtained by adding the sets $O^g=\{T\in\mathcal{T}(P):g\C P\setminus T\}$ to the subbasis.  This answers a question posed by Gilles de Castro.
\end{rmk}

\begin{rmk}
Note $G\C P\setminus T$ implies $G\subseteq P\setminus T$, but not conversely.  That is unless $\prec$ is reflexive and hence a preorder \textendash\, then $G\C P\setminus T$ is indeed the same as $G\subseteq P\setminus T$ and the topology in this case could be viewed as coming from the product topology on the characteristic functions of the tight subsets.  Thus when $(P,\prec)=(S\setminus\{0\},\leq)$, for some $\wedge$-semilattice $S$ with minimum $0$, our tight spectrum agrees with Exel's original tight spectrum \textendash\, see \cite{ExelPardo2016}.
\end{rmk}

\begin{rmk}
There is also actually no need for $G$ above to be finite.  Indeed, for any $Q\subseteq P$, the definition of $\C$, \eqref{DC=>C} and \eqref{FprecQ} shows that $Q\C P\setminus T$ is equivalent to $Q\D G\C P\setminus T$, for some finite $G$.  Thus $O^Q$ is open as a union of open sets, i.e.
\[O^Q=\{T\in\mathcal{T}(P):Q\C P\setminus T\}=\bigcup_{\substack{G\text{ is finite}\\\text{and }Q\,\D\,G}}O^G.\]
\end{rmk}

Now we consider topological properties of the tight spectrum.

\begin{prp}\label{TightHausdorff}
The tight spectrum is Hausdorff.
\end{prp}

\begin{proof}
Take distinct $T,U\in\mathcal{T}(P)$.  Thus we must have $t\in T\setminus U$ (or vice versa).  As $T$ is round, we have $s\in T$ with $s\prec t$, which means $T\in O_s$ and $U\in O^s$.  If $s\in V\in\mathcal{T}(P)$ then $V\not\CC P\setminus V$ and hence $s\not\C P\setminus V$ so $O_s\cap O^s=\emptyset$.
\end{proof}

Note the above proof used the fact that $O^s_s=\emptyset$.  More generally, for any finite $F,G\subseteq P$, we can characterize when $O_F^G=\emptyset$ as follows.

\begin{prp}\label{OFGempty}
$O_F^G=\emptyset$ iff every $\C$-cover of $G$ is a $\C$-cover of $\widehat{F}$, i.e.
\begin{equation}\label{C=>L}
G\C H\qquad\Rightarrow\qquad\widehat{F}\C H.
\end{equation}
\end{prp}

\begin{proof}
Assume $G\C H$ and $\widehat{F}\not\C H$.  By \eqref{RCQ}, we have finite $G'$ with $G\C G'\prec H$ and hence $\widehat{F}\not\D G'$.  This means we have $p$ with $F\succ p\perp G'$.  As $P$ is round, we have a sequence $p=p_1\succ p_2\succ\ldots$.  By the Kuratowski-Zorn lemma, this sequence extends to some maximal round centred $T$, necessarily with $F\subseteq T$ and $T\cap G'=\emptyset$.  By \eqref{TightForall}, $T$ is tight so $T\in O_F^G$, proving the `only if' part.

Conversely, if $T\in O_F^G$ then $F\subseteq T$ and $G\C P\setminus T$.  As $T$ is tight, $F\subseteq T\not\CC P\setminus T$ so $\widehat{F}\not\C P\setminus T$, proving the `if' part.
\end{proof}

This has several important corollaries, e.g. we can generalise \cite[Theorem 12.9]{Exel2008} which says that ultrafilters are dense in the tight spectrum of a $\wedge$-semilattice.

\begin{cor}
The maximal round centred subsets of $P$ are dense in $\mathcal{T}(P)$.
\end{cor}

\begin{proof}
Any non-empty open set contains some non-empty $O_F^G$.  Now just note that the $T\in O_F^G$ in the proof of \autoref{OFGempty} is a maximal round centred subset.
\end{proof}

\begin{cor}
For all $p,q\in P$
\begin{equation}\label{pperpq}
p\perp q\qquad\Leftrightarrow\qquad O_p\cap O_q=\emptyset.
\end{equation}
\end{cor}

\begin{proof}
If $p\perp q$ then no centred set can contain both $p$ and $q$ so $O_p\cap O_q=\emptyset$.  Conversely, if $O_{p,q}^\emptyset=O_p\cap O_q=\emptyset$ then, as $\emptyset\C\emptyset$, we must have $\widehat{\{p,q\}}\C\emptyset$, by \autoref{OFGempty}.  Thus $\widehat{\{p,q\}}=\emptyset$, by \eqref{CMin}, i.e. $p\perp q$.
\end{proof}

In particular, for any $p\in P$, we have $p\not\perp p$, as $P$ is round, so \eqref{pperpq} yields
\begin{equation}\label{Opnonempty}
O_p\neq\emptyset.
\end{equation}

\begin{cor}\label{FDG=>nonempty}
If $\widehat{F}\D G$ then $O_F^G=\emptyset$.
\end{cor}

\begin{proof}
If $\widehat{F}\D G\C H$ then $\widehat{F}\C H$, by \eqref{DC=>C}, so $O_F^G=\emptyset$, by \autoref{OFGempty}.
\end{proof}

Conversely, if $O_F^G=\emptyset$ and $P$ is a poset then $G$ itself is a $\C$/$\D$-cover of $G$ so $\widehat{F}\D G$, by \autoref{OFGempty}.

\begin{xpl}
This no longer holds when $P$ is not a poset.  For example, let $P=\mathbb{N}\cup\{a\}$ and let $\prec$ be $=$ on $\mathbb{N}$ while $n\prec a$ iff $n\neq2$ and $a\not\prec a$.  Then $2\perp a$ and, in particular, $2\not\D a$ even though $a$ has no $\C$-cover and hence $O^a_{2}=\emptyset$.
\end{xpl}

We can now show that $\D$ has the desired representation in $\mathcal{T}(P)$.

\begin{prp}
For any $F,Q\subseteq P$ with $F$ finite,
\begin{equation}\label{FDCQ}
\widehat{F}\D Q\qquad\Leftrightarrow\qquad O_F\subseteq\overline{\bigcup_{q\in Q}O_q}.
\end{equation}
\end{prp}

\begin{proof}
If $\widehat{F}\not\D Q$ then we have $p\in P$ with $F\succ p\perp Q$, which means $O_F\supseteq O_p$ and $O_p\cap\bigcup_{q\in Q}O_q=\emptyset$, by \eqref{pperpq}.  By \eqref{Opnonempty}, $\emptyset\neq O_p\subseteq O_F\setminus\overline{\bigcup_{q\in Q}O_q}$ so $O_F\not\subseteq\overline{\bigcup_{q\in Q}O_q}$.

Conversely, if $O_F\not\subseteq\overline{\bigcup_{q\in Q}O_q}$ then we have $T\in O^G_H\subseteq O_F\setminus\overline{\bigcup_{q\in Q}O_q}$, for some finite $G,H\subseteq P$.  As $G\C P\setminus T$, \eqref{RCQ} yields $I$ with $G\C I\prec P\setminus T$.  As $T$ is tight, $F\cup H\subseteq T$ and $I\prec P\setminus T$, we have some $p\in P$ with $F\cup H\succ p\perp I$.  Thus $O_p\subseteq O^G_H$ and hence $O_p\cap\bigcup_{q\in Q}O_q=\emptyset$ so $p\perp Q$, by \eqref{pperpq}.  As we also have $F\succ p$, this shows that $\widehat{F}\not\D Q$.
\end{proof}

\begin{cor}\label{RDQ}
For any $Q,R\subseteq P$,
\[R\D Q\qquad\Leftrightarrow\qquad\bigcup_{r\in R}O_r\subseteq\overline{\bigcup_{q\in Q}O_q}.\]
\end{cor}

\begin{proof}
Note $R\D Q$ iff $r\D Q$, for all $r\in R$.  By \eqref{FDCQ}, this is the same as saying $O_r\subseteq\overline{\bigcup_{q\in Q}O_q}$, for all $r\in R$, i.e. $\bigcup_{r\in R}O_r\subseteq\overline{\bigcup_{q\in Q}O_q}$.
\end{proof}

Next we wish to do the same for $\C$.  First we need the following.

\begin{prp}\label{TinClosureO_F}
For any finite $F\subseteq P$ and $T\in\mathcal{T}(P)$,
\begin{equation}\label{TinClosure}
T\in\overline{O_F}\qquad\Leftrightarrow\qquad T\cup F\not\CC P\setminus T.
\end{equation}
\end{prp}

\begin{proof}
Note $T\notin\overline{O_F}$ iff we have a basic neighbourhood of $T$ disjoint from $O_F$, i.e. iff we have finite $F'\subseteq T$ and $G\C P\setminus T$ with $O_{F\cup F'}^G=O_F\cap O_{F'}^G=\emptyset$.  By \eqref{C=>L}, this implies $\widehat{F'\cup F}\C P\setminus T$ and hence $T\cup F\CC P\setminus T$.

Conversely, if $T\cup F\CC P\setminus T$ then we have finite $F'\subseteq T$ with $\widehat{F'\cup F}\C P\setminus T$ and hence we have finite $G$ with $\widehat{F'\cup F}\D G\prec P\setminus T$.  So $G\C P\setminus T$, by \eqref{FprecQ}, and $O_F\cap O_{F'}^G=O_{F\cup F'}^G=\emptyset$, by \autoref{FDG=>nonempty}.  Thus $T\in O_{F'}^G$ and hence $T\notin\overline{O_F}$.
\end{proof}

Let $\Subset$ denote the compact containment relation, i.e.
\[O\Subset N\qquad\Leftrightarrow\qquad\exists\text{ compact }C\ (O\subseteq C\subseteq N).\]
So in Hausdorff spaces, $O\Subset N$ is just saying that $\overline{O}$ is compact and $\overline{O}\subseteq N$.

\begin{thm}\label{O_FSubset}
For any $F,Q\subseteq P$ with $F$ finite,
\begin{equation}\label{FCCQ}
\widehat{F}\C Q\qquad\Leftrightarrow\qquad O_F\Subset\bigcup_{q\in Q}O_q.
\end{equation}
\end{thm}

\begin{proof}
Assume $F$ is finite and $\widehat{F}\C Q$.  Then, for any $T\in\mathcal{T}(P)\setminus\bigcup_{q\in Q}O_q$, we have $\widehat{F}\C Q\subseteq P\setminus T$ so $T\notin\overline{O_F}$, by \autoref{TinClosureO_F}.  Thus $\overline{O_F}\subseteq\bigcup_{q\in Q}O_q$.

To prove that $\overline{O_F}$ is compact, it suffices to show that every subbasic open cover has a finite subcover, by the Alexander subbasis theorem.  So assume we have some $S\subseteq P$ and a collection $\mathcal{G}$ of finite subsets of $P$ such that no finite subcollection of $(O_s)_{s\in S}$ and $(O^G)_{G\in\mathcal{G}}$ covers $\overline{O_F}$ (in particular we must have $O_F\neq\emptyset$ and hence $\widehat{F}\neq\emptyset$).  We can also assume that $\mathcal{G}$ contains some $G$ with $G\C H$, for some finite $H$.  Indeed, we have $G$ with $\widehat{F}\C G\prec Q$ and hence $\emptyset\neq G\prec H\neq\emptyset$ for some finite $H\subseteq Q$.  Thus $\overline{O_F}\cap O^G=O_F^G=\emptyset$, which means we could add $G$ to $\mathcal{G}$ and there would still be no finite subcollection covering $\overline{O_F}$.

We need to show that the entire collection of $(O_s)_{s\in S}$ and $(O^G)_{G\in\mathcal{G}}$ does not cover $\overline{O_F}$ either.  For this we first apply the selection principle in \autoref{SelectionPrinciple} with
\begin{align*}
\Delta&=\{H\subseteq P:H\text{ is finite and }\widehat{H\cup F}\not\C S\}.\\
\Gamma&=\{H\subseteq P:H\text{ is finite and }\mathcal{G}\ni G\C H\}.
\end{align*}
We immediately see that $\Delta$ satisfies \eqref{succClosed}.  Also $\Gamma$ satisfies \eqref{precRound}, by \eqref{RCQ}.  Moreover, we can assume $\Gamma$ contains some non-empty $H$, as mentioned above.  To see that \eqref{DeltaCentred} holds, take finite $\Phi\subseteq\Gamma$.  For each $H\in\Phi$, we have $G_H\in\mathcal{G}$ with $G_H\C H$.  For every finite $I\subseteq S$, the fact that $\overline{O_F}$ has no finite cover in $(O_s)_{s\in S}$ and $(O^G)_{G\in\mathcal{G}}$ yields
\[T\in\overline{O_F}\setminus\bigcup_{s\in I}O_s\cup\bigcup_{H\in\Phi}O^{G_H}.\]
For each $H\in\Phi$, the fact that $T\notin O^{G_H}$ and $G_H\C H$ means $H\not\subseteq P\setminus T$, i.e. $T\cap H\neq\emptyset$.  Thus $D=\bigcup_{H\in\Phi}T\cap H$ satisfies $D\cap H\neq\emptyset$, for each $H\in\Phi$.  By \eqref{TinClosure},
\[D\cup F\subseteq T\cup F\not\CC P\setminus T\supseteq I.\]
As this holds for all finite $I\subseteq S$, we have $\widehat{D\cup F}\not\C S$ and hence $D\in\Delta$.  This verifies \eqref{DeltaCentred}, as required.

Applying \autoref{SelectionPrinciple} we obtain $R\subseteq P$ such that $H\cap R\neq\emptyset$, for all $H\in\Gamma$, and $D\in\Delta$, for all finite $D\subseteq R$, which means $R\cup F\not\CC S$.  By \autoref{TightStretching}, we have tight $T\supseteq R$ with $T\cup F\not\CC P\setminus T\supseteq S$.  Thus $T\neq\emptyset$, as $\Gamma\neq\emptyset$, and $T\in\overline{O_F}$, by \eqref{TinClosure}.  Moreover $T\notin\bigcup_{s\in S}O_s$, as $S\subseteq P\setminus T$, and $T\notin\bigcup_{G\in\mathcal{G}}O^G$, as $H\cap T\neq\emptyset$, for all $H$ with $G\C H$, for some $G\in\mathcal{G}$.  Thus $T\in\overline{O_F}\setminus(\bigcup_{s\in S}O_s\cup\bigcup_{G\in\mathcal{G}}O^G)$ and hence the entirety of $(O_s)_{s\in S}$ and $(O^G)_{G\in\mathcal{G}}$ does not cover $\overline{O_F}$ either.  This shows that $\overline{O_F}$ is indeed compact and hence $O_F\Subset\bigcup_{q\in Q}O_q$.

Conversely, assume $\widehat{F}\not\C Q$.  If, moreover, $\widehat{F}\C P$ then, by \eqref{RCQ}, we have a sequence $(G_n)$ of finite subsets of $P$ such that, for all $n$,
\[\widehat{F}\C G_{n+1}\prec G_n\prec P.\]
Using \autoref{CentredDense} and \autoref{SelectionPrinciple} (or K\"onig's lemma), we obtain round $R$ with $F\cup R\not\CC Q$ and $R\cap G_n\neq\emptyset$, for each $n$.  In particular, $R\neq\emptyset$.  By \autoref{TightStretching}, we have tight $T\supseteq R$ with $F\cup T\not\CC P\setminus T\supseteq Q$.  By \eqref{TinClosure}, this means $T\in\overline{O_F}\setminus\bigcup_{q\in Q}O_q$ and hence $O_F\not\Subset\bigcup_{q\in Q}O_q$.

On the other hand, if $\widehat{F}\not\C P$ then $\overline{O_F}$ is not even compact so $O_F\not\Subset\bigcup_{q\in Q}O_q$.  To see this, note that $(O_p)_{p\in P^\succ}$ covers the entire tight spectrum, as each $T\in\mathcal{T}(P)$ is non-empty and round.  However, for any finite $G\prec P$, $\widehat{F}\not\C P$ yields $\widehat{F}\not\D G$ so we have $f$ with $F\succ f\perp G$.  By \eqref{Opnonempty}, we have $T\in O_f\subseteq O_F$ and hence $T\notin\bigcup_{g\in G}O_g$, by \eqref{pperpq}.  As $G$ was arbitrary, this shows $(O_p)_{p\in P^\succ}$ has no finite subcover of $O_F$.
\end{proof}

\begin{cor}\label{Crep}
For any $Q,R\subseteq P$,
\[R\C Q\qquad\Leftrightarrow\qquad\bigcup_{r\in R}O_r\Subset\bigcup_{q\in Q}O_q.\]
\end{cor}

\begin{proof}
If $R\C Q$ then $R\D F\prec Q$, for some finite $F\subseteq P$.  Then \eqref{RDQ} yields $\bigcup_{r\in R}O_r\subseteq\overline{\bigcup_{f\in F}O_f}$, while \eqref{FprecQ} and \eqref{FCCQ} yield $O_f\Subset\bigcup_{q\in Q}O_q$, for all $f\in F$.  Thus $\overline{\bigcup_{f\in F}O_f}=\bigcup_{f\in F}\overline{O_f}$ is compact and $\bigcup_{r\in R}O_r\subseteq\overline{\bigcup_{f\in F}O_f}\subseteq\bigcup_{q\in Q}O_q$, i.e.
\[\bigcup_{r\in R}O_r\Subset\bigcup_{q\in Q}O_q.\]

Conversely, if $\bigcup_{r\in R}O_r\Subset\bigcup_{q\in Q}O_q$ then, for each $T\in\overline{\bigcup_{r\in R}O_r}\subseteq\bigcup_{q\in Q}O_q$, we have some $t\in T\cap Q^\succ$.  Thus $\overline{\bigcup_{r\in R}O_r}\subseteq\bigcup_{t\in Q^\succ}O_t$ and, by compactness, we have some finite $F\prec Q$ with $\overline{\bigcup_{r\in R}O_r}\subseteq\bigcup_{f\in F}O_f$ and hence $\bigcup_{r\in R}O_r\subseteq\overline{\bigcup_{f\in F}O_f}$.  Thus $R\D F\prec Q$, by \eqref{RDQ}, i.e. $R\C Q$.
\end{proof}

Recall that a Hausdorff space is \emph{locally compact} iff each point has a compact neighbourhood (which implies that each point actually has a neighbourhood base of compact sets and is thus consistent with \autoref{LCdefinitions} below).

\begin{cor}\label{TightLocallyCompact}
The tight spectrum is locally compact.
\end{cor}

\begin{proof}
Any $T\in\mathcal{T}(P)$ is non-empty and round, so we have $s,t\in T$ with $s\prec t$.  By \eqref{FprecQ}, $s\C t$ so $T\in O_s\Subset O_t$, by \autoref{O_FSubset}.  In particular, $T$ has a compact neighbourhood so, as $T$ was arbitrary, $\mathcal{T}(P)$ is locally compact.
\end{proof}

For $\mathcal{T}(P)$ to be locally compact as above, \eqref{Shrinking} is crucial, as the following example shows (this answers a question posed by Exel).

\begin{xpl}\label{ExelQuestion}
Let $P$ be the collection of all dyadic open intervals in $\mathbb{R}$, i.e.
\[P=\{(\tfrac{k}{2^n},\tfrac{k+1}{2^n}):k,n\in\mathbb{Z}\}.\]
Considering $(P,\Subset)$, we see that $\Subset$ is transitive and $P$ is round but \eqref{Shrinking} fails so $(P,\Subset)$ is not an abstract pseudobasis.  Moreover, $\mathcal{T}(P)$ can be identified with the complement $\mathbb{R}\setminus D$ of the dyadic numbers $D\subseteq\mathbb{R}$.  As both $D$ and $\mathbb{R}\setminus D$ are dense in $\mathbb{R}$, all compact subsets of $\mathbb{R}\setminus D$ have empty interior and hence $\mathcal{T}(P)\approx\mathbb{R}\setminus D$ is not locally compact.
\end{xpl}

\subsection{Patch Topologies}
We are primarily concerned with pseudobases (to be introduced in the next section) in locally compact Hausdorff spaces.  However, it will be useful to first consider more general pseudosubbases and how they relate to more general stably locally compact $T_0$ spaces via the patch construction.

First, for any topological space $X$, let $\mathcal{O}(X)$ denote the open subsets of $X$ and let $\mathcal{C}(X)$ denote the compact saturated subsets of $X$, i.e. the compact subsets $C\subseteq X$ that are also intersections of open sets so
\[C=\bigcap_{C\subseteq O\in\mathcal{O}(X)}O.\]
Note that if $X$ is Hausdorff or even $T_1$ then arbitrary subsets are saturated so $\mathcal{C}(X)$ in this case just consists of arbitrary compact subsets.

\begin{dfn}\label{LCdefinitions}
A topological space $X$ is
\begin{enumerate}
\item \emph{locally compact} if each point has a neighbourhood base of compact sets.
\item \emph{coherent} if $X$ is locally compact and $\mathcal{C}(X)$ is closed under $\cap$, i.e.
\[\tag{$\cap$-Closed}O,N\in\mathcal{C}(X)\quad\Rightarrow\quad O\cap N\in\mathcal{C}(X).\]
\item \emph{stably locally compact} if $X$ is coherent and, for $O\in\mathcal{O}(X)$ and $\mathcal{C}\subseteq\mathcal{C}(X)$
\[\label{WellFiltered}\tag{Well-Filtered}\bigcap\mathcal{C}\subseteq O\quad\Rightarrow\quad\exists\text{ finite }F\subseteq\mathcal{C}\ (\bigcap F\subseteq O).\]
\end{enumerate}
\end{dfn}

\begin{rmk}\label{StablyLocallyCompactSubset}
These are the standard definitions \textendash\, see \cite[Definitions 4.8.1, 5.2.21 and 8.3.31]{Goubault2013} and \cite[Definitions O-5.9, VI-6.2 and VI-6.7]{GierzHofmannKeimelLawsonMisloveScott2003}.  Equivalent definitions can be stated using $\Subset$ on $\mathcal{O}(X)$ instead, specifically $X$ is
\begin{enumerate}
\item \emph{locally compact} if $\{O\in\mathcal{O}(X):x\in O\}$ is $\Subset$-round, for each $x\in X$.
\item \emph{stably locally compact} if, for $O,O',N,N'\in\mathcal{O}(X)$ and $\Subset$-round $\mathcal{O}\subseteq\mathcal{O}(X)$,
\begin{align*}
O'\Subset O\ \text{and}\ N'\Subset N\quad&\Rightarrow\quad O'\cap N'\Subset O\cap N,\quad\text{and}\\
\bigcap\mathcal{O}\subseteq O\quad&\Rightarrow\quad\exists\text{ finite }F\subseteq\mathcal{O}\ (\bigcap F\subseteq O).
\end{align*}
\end{enumerate}
(see \cite[Proposition VI-7.71]{GierzHofmannKeimelLawsonMisloveScott2003} and \cite[Theorem 8.3.34]{Goubault2013}).
Also, among locally compact $T_0$ spaces, the directed version of \eqref{WellFiltered} is equivalent to being sober \textendash\, see \cite[Propositions 8.3.5 and 8.3.8]{Goubault2013}.
\end{rmk}

In a stably locally compact space $X$, both $\mathcal{C}(X)$ and $\Subset$ behave like they would in a Hausdorff space.  Indeed, in this case $\mathcal{C}(X)$ and $\Subset$ are stable under the patch construction, which makes $X$ Hausdorff, as long as $X$ was originally $T_0$.

Specifically, recall that the \emph{patch topology} of a topological space is the topological join of $\mathcal{O}(X)$ and $\mathcal{C}(X)^c$, i.e. the topology generated by all open sets and all complements of compact saturated sets (see \cite[Definition V-5.11]{GierzHofmannKeimelLawsonMisloveScott2003} or \cite[Definition 9.1.26]{Goubault2013}).  We denote $X$ with its patch topology by $X^\mathrm{patch}$ so
\[\mathcal{O}(X^\mathrm{patch})=\mathcal{O}(X)\vee\mathcal{C}(X)^c.\]

\begin{prp}
If $X$ is a stably locally compact $T_0$ space then $X^\mathrm{patch}$ is a locally compact Hausdorff space such that $\mathcal{C}(X)\subseteq\mathcal{C}(X^\mathrm{patch})$ and, for all $O,N\in\mathcal{O}(X)$,
\begin{equation}\label{SubsetPatch}
O\Subset N\text{ in }X\qquad\Leftrightarrow\qquad O\Subset N\text{ in }X^\mathrm{patch}.
\end{equation}
\end{prp}

\begin{proof}
First note that the patch topology of any locally compact $T_0$ space is Hausdorff, by \cite[Proposition 9.1.12 and Lemma 9.1.31]{Goubault2013}.

For any $C,D\in\mathcal{C}(X)$, $C\cap D\in\mathcal{C}(X)$, as $X$ is coherent.  Thus $C^\mathrm{patch}$, the patch topology coming from the subspace topology of $C$, coincides with the subspace topology coming from $X^\mathrm{patch}$.  By \cite[Proposition 9.1.27]{Goubault2013}, $C^\mathrm{patch}$ is compact so $C$ is compact in $X^\mathrm{patch}$ and hence $C\in\mathcal{C}(X^\mathrm{patch})$, as $X$ is Hausdorff.

As $X$ is locally compact, every $x\in X$ has a neighbourhood in $\mathcal{C}(X)\subseteq\mathcal{C}(X^\mathrm{patch})$.  As $X^\mathrm{patch}$ is Hausdorff, it follows that $X^\mathrm{patch}$ is also locally compact.

Now take any $O,N\in\mathcal{O}(X)$.  If $O\Subset N$ in $X^\mathrm{patch}$ then this remains true in $X$ because any compact subset of $X^\mathrm{patch}$ is certainly compact in the coarser original topology of $X$.  Conversely, if $O\Subset N$ in $X$ then we have compact $C$ in $X$ with $O\subseteq C\subseteq N$.  But the saturation of $C$ is also compact, by \cite[Proposition 4.4.14]{Goubault2013}, and is still contained in the open set $N$.  Thus we may take $C\in\mathcal{C}(X)\subseteq\mathcal{C}(X^\mathrm{patch})$ and hence $O\Subset N$ in $X^\mathrm{patch}$.
\end{proof}

Conversely, to produce general stably locally compact $T_0$ spaces from locally compact Hausdorff spaces, we consider pseudosubbases.

\begin{dfn}
We call $P\subseteq\mathcal{O}(X)$ a \emph{pseudosubbasis} of a space $X$ if
\begin{align}
\label{Separating'}\tag{Separating}&\forall x,y\in X\ \exists O\in P\ (x\neq y\quad\Rightarrow\quad x\notin O\ni y\ \ \text{or}\ \ y\notin O\ni x).\\
\label{capPointRound}\tag{$\cap$-Point-Round}&\forall\text{ finite }F\subseteq P\ \forall x\in\bigcap F\ \exists\text{ finite }G\subseteq P\ (x\in\bigcap G\Subset\bigcap F).
\end{align}
\end{dfn}

\begin{rmk}
Recall that a \emph{subbasis} $S$ of a space $X$ is a collection of open sets that generates the topology of $X$, i.e. such that every element of $\mathcal{O}(X)$ is a union of finite intersections of elements of $S$.  In particular, every subbasis of a locally compact $T_0$ space is a pseudosubbasis.  Conversely, if $X$ has a pseudosubbasis then $X$ must be both locally compact and $T_0$.
\end{rmk}

\begin{rmk}
By \eqref{Separating}, any pseudosubbasis $P$ must cover all except possibly one point of $X$, i.e. $|X\setminus\bigcup P|\leq 1$.  Taking $F=\emptyset$ in \eqref{capPointRound} we see that, for any $x\in X=\bigcap\emptyset$, we have finite $G\subseteq P$ with $x\in\bigcap G\Subset X$.  If $X$ is compact then we can always take $G=\emptyset$ too but otherwise $G$ must be non-empty, i.e. if $X$ is not compact then $P$ must truly cover $X$.
\end{rmk}

\begin{prp}
Any pseudosubbasis of a stably locally compact space $X$ is a pseudosubbasis of $X^\mathrm{patch}$.
\end{prp}

\begin{proof}
Immediate because $\Subset$ is the same on $\mathcal{O}(X)$ and $\mathcal{O}(X^\mathrm{patch})$, by \eqref{SubsetPatch}.
\end{proof}

In particular, any true subbasis of a stably locally compact $T_0$ space $X$ is a pseudosubbasis of the locally compact Hausdorff space $X^\mathrm{patch}$.  Conversely, we now show that any pseudosubbasis of a locally compact Hausdorff space generates a stably locally compact $T_0$ topology from which the original topology can be recovered from the patch construction.  Thus we have the following duality:
\begin{gather}
\nonumber\text{pseudosubbases of locally compact Hausdorff spaces}\\
\label{pseudosubDuality}\updownarrow\\
\nonumber\text{subbases of stably locally compact $T_0$ spaces}.
\end{gather}

Given $P\subseteq\mathcal{P}(X)$, let $X_P$ denote $X$ with the topology generated by $P$.

\begin{prp}\label{PseudoPatch}
If $P$ is a pseudosubbasis of a locally compact Hausdorff space $X$ then $X_P$ is a stably locally compact $T_0$ space such that $X=X_P^\mathrm{patch}$.
\end{prp}

\begin{proof}
First replace $P$ with its closure under finite intersections.  Note $P$ is still a pseudosubbasis of $X$ and now also a basis of $X_P$.

We claim that $\mathcal{C}(X_P)\subseteq\mathcal{C}(X)$.  To see this, take any $C\in\mathcal{C}(X_P)$ and $x\in X\setminus C$.  As $C$ is saturated in $X_P$, we have $O\in\mathcal{O}(X_P)$ with $x\notin O\supseteq C$.  For every $c\in C$, we have $N\in P$ with $c\in N\subseteq O$.  As $C$ is compact in $X_P$, we thus have finite $F\subseteq P$ with $C\subseteq\bigcup F\subseteq O$.  As $x$ was arbitrary, this shows that $C=\bigcap_{F\in\Gamma}\bigcup F$, where $\Gamma$ is the collection of all finite $F\subseteq P$ with $C\subseteq\bigcup F$.  Again by compactness in $X_P$ and \eqref{capPointRound}, $\Gamma$ is $\Subset$-round (where $\Subset$ is defined from the original topology of $X$), i.e. for all $F\in\Gamma$, we have $G\in\Gamma$ with $G\subseteq F^\Supset$.  As $X$ is Hausdorff this means $C=\bigcap_{F\in\Gamma}\bigcup F=\bigcap_{F\in\Gamma}\overline{\bigcup F}$, where the closure is taken in $X$.  Thus $C$ is an intersection of compact subsets in the Hausdorff space $X$ so $C\in\mathcal{C}(X)$.

As in the proof of \eqref{SubsetPatch}, this means that, for any $O,N\in\mathcal{O}(X_P)$,
\begin{equation}
O\Subset N\text{ in }X_P\qquad\Leftrightarrow\qquad O\Subset N\text{ in }X.
\end{equation}
As stable local compactness can be characterised by $\Subset$ as in \autoref{StablyLocallyCompactSubset}, and $X$ is certainly stably locally compact (being Hausdorff), it follows that $X_P$ is also stably locally compact.  By \eqref{Separating}, $X_P$ is also $T_0$.

Thus $X_P^\mathrm{patch}$ is Hausdorff, by \cite[Proposition 9.1.12 and Lemma 9.1.31]{Goubault2013}.  As $\mathcal{C}(X_P)\subseteq\mathcal{C}(X)$, the identity $\mathbf{id}:X\rightarrow X_P^\mathrm{patch}$ is continuous.  For any $x\in X$, \eqref{capPointRound} yields $O,N\in P$ with $x\in N\Subset O$.  As $\overline{N}$ is compact and $X_P^\mathrm{patch}$ is Hausdorff, $\mathbf{id}$ takes closed subsets of $\overline{N}$ to closed subsets, i.e. $\mathbf{id}$ is a homeomorphism on $\overline{N}$ and hence on $N$.  As $N$ is also open in $X_P^\mathrm{patch}$, $\mathbf{id}$ is a local homeomorphism.  But $\mathbf{id}$ is clearly bijective and thus a global homeomorphism.
\end{proof}

\begin{rmk}
There is also an ordering induced on $X$ by $P$, namely
\[x\leq y\qquad\Leftrightarrow\qquad\forall O\in P\ (x\in O\ \Rightarrow\ y\in O).\]
As long as $P$ is a pseudosubbasis of a locally compact Hausdorff space, one can verify that $\leq$ is closed in $X\times X$ and thus turns $X$ into a pospace.  Moreover, the open up-sets in $X$ are then precisely the open sets in $X_P$ (and $\leq$ is just the specialisation order in $X_P$).  By \eqref{capPointRound}, the open up-sets $N^\uparrow_x$ containing any $x\in X$ will also be $\Subset$-round.  Thus an alternative way of viewing \eqref{pseudosubDuality} is
\begin{gather*}
\text{locally compact pospaces such that each $N_x^\uparrow$ is $\Subset$-round}\\
\updownarrow\\
\text{stably locally compact $T_0$ spaces}.
\end{gather*}
This generalises the duality between compact pospaces and stably compact spaces (see \cite[Proposition 9.1.27 and 9.1.34]{Goubault2013} and \cite[Proposition VI-6.23]{GierzHofmannKeimelLawsonMisloveScott2003}).
\end{rmk}

\subsection{Concrete Pseudobases}\label{ConcretePseudobases}

We have just seen in \autoref{PseudoPatch} that the topology of any locally compact Hausdorff space $X$ can be recovered from any pseudosubbasis.  However, this requires that we already know the points of $X$.  What we would like to do is recover both the points and the topology from the relational structure $(P,\Subset)$, for suitable $P\subseteq\mathcal{O}(X)$.  Pseudosubbases are too general for this as the following elementary examples show.

\begin{xpl}
Consider the following possibilities.
\begin{align*}
X&=\{x,y\}&P&=\{\emptyset,\{x\},\{y\},\{x,y\}\}.\\
X&=\{x,y,z\}&P&=\{\emptyset,\{x\},\{y\},\{x,y,z\}\}.
\end{align*}
In both cases, $P$ is a (even $\cap$-closed) pseudosubbasis for $X$, considered as a discrete space.  Moreover, the resulting relational structures $(P,\Subset)$ are immediately seen to be isomorphic, even though the spaces are certainly not homeomorphic, having different numbers of points.
\end{xpl}

This leads us to consider the following stronger structures.  Note \eqref{Dense} is the key extra condition distinguishing pseudobases from pseudosubbases.

\begin{dfn}
We call $P\subseteq\mathcal{O}(X)$ a \emph{(concrete) pseudobasis} of a space $X$ if
\begin{align}
\label{Cover}\tag{Cover}&\text{Every $x\in X$ is contained in some $O\in P$}.\\
\label{Separating}\tag{Separating}&\text{The subsets in $P$ distinguish the points of $X$}.\\
\label{PointRound}\tag{Point-Round}&\text{The neighborhoods in $P$ of any fixed $x\in X$ are $\Subset$-round}.\\
\label{Dense}\tag{Dense}&\text{Every non-empty $O\in\mathcal{O}(X)$ contains some non-empty $N\in P$}.
\end{align}
\end{dfn}

Our primary goal in this section is to exhibit a duality between abstract and concrete pseudobases.  First, however, we make some general observations.

To begin with, pseudobases are indeed generalisations of bases, at least when $X$ is locally compact and $T_0$.  Indeed, bases of locally compact spaces satisfy \eqref{PointRound}, while bases of $T_0$ spaces satisfy \eqref{Separating} and, by definition, bases of arbitrary spaces satisfy \eqref{Cover} and \eqref{Dense}.

\begin{xpl}
Consider the one-point compactification $X=\mathbb{N}\cup\{\infty\}$ of $\mathbb{N}$ and let $P=\{\{n\}:n\in\mathbb{N}\}\cup\{X\}$.  It is immediately verified that $P$ is a pseudobasis but not a basis, as it does not contain a neighbourhood base of the point $\infty$.
\end{xpl}

\begin{rmk}
Our previous notion of a pseudobasis from \cite[Definition 8.2]{BiceStarling2016} contained all the above axioms except \eqref{PointRound}.  Indeed, \eqref{PointRound} holds trivially for the compact pseudobases considered in \cite{BiceStarling2016}.  One immediate consequence of \eqref{PointRound} worth noting is that $\subseteq$ coincides with the lower order defined from $\Subset$, i.e. for all $O,N\in P$,
\[\tag{Lower Order}O\subseteq N\qquad\Leftrightarrow\qquad\forall M\in P\ (M\Subset O\ \ \Rightarrow\ \ M\Subset N).\]
\end{rmk}

\begin{rmk} there are several equivalent ways of stating the pseudobasis axioms.  For example, $P\subseteq\mathcal{O}(X)$ is a pseudobasis if and only if the following are satisfied.
\begin{align}
\label{Cover+Round}\tag{Cover+Point-Round}O&=\bigcup_{P\ni N\Subset O}N,\text{ for all }O\in P\cup\{X\}.\\
\label{Dense'}\tag{Dense}\overline{O}&=\overline{\bigcup_{P\ni N\Subset O}N},\text{ for all }O\in\mathcal{O}(X).\\
\tag{Separating}\{x\}&=\bigcap_{x\in O\in P}O\cap\bigcap_{x\notin O\in P}X\setminus O,\text{ for all }x\in X.
\end{align}
We will primarily be concerned with pseudobases of non-empty sets.  In this case, we can set $N_x=\{O\in P:x\in O\}$ and concisely state the pseudobasis axioms as
\begin{align}
\tag{Non-Empty Open+Dense}&\emptyset\notin P\subseteq\mathcal{O}(X)\subseteq P^\Subset.\\
\tag{Cover+Separating+Round}&\emptyset\neq N_x\neq N_y\subseteq N_y^\Subset,\quad\text{for all distinct $x,y\in X$}.
\end{align}
\end{rmk}

Now for the promised duality between abstract and concrete pseudobases.  First we represent abstract pseudobases as concrete pseudobases.

\begin{prp}\label{abstract->concrete}
If $(P,\prec)$ is an abstract pseudobasis then $(O_p)_{p\in P}$ is a concrete pseudobasis of $\mathcal{T}(P)$.
\end{prp}

\begin{proof}\
\begin{itemize}
\item[\eqref{Cover}]  As every $T\in\mathcal{T}(P)$ is non-empty, we have some $t\in T$ and hence $T\in O_t$.

\item[\eqref{PointRound}]  If $T\in O_t$ then $t\in T$ so, as $T$ is round, we have $s\in T$ with $s\prec t$ and hence $T\in O_s\Subset O_t$, by \eqref{FprecQ} and \autoref{O_FSubset}.

\item[\eqref{Dense}]  If $T\in O_F^G$ then $F\subseteq T$ and $G\C S\setminus T$.  By \eqref{RCQ}, we have $H$ with $G\C H\prec P\setminus T$.  By \eqref{Tight-Round}, we have $p$ with $F\succ p\perp H$ so $O_p\subseteq O_F^G$.  As the $(O_F^G)$ form a basis for $\mathcal{T}(P)$, this shows that any non-empty open $O\subseteq\mathcal{T}(P)$ contains some $O_p$, which is non-empty by \eqref{Opnonempty}.

\item[\eqref{Separating}]  If $T,U\in\mathcal{T}(P)$ are distinct then we have some $t\in T\setminus U$ or $t\in U\setminus T$, i.e. $U\notin O_t\ni T$ or $T\notin O_t\ni U$ so $O_t$ distinguishes $T$ from $U$. \qedhere
\end{itemize}
\end{proof}

Next we show how concrete pseudobases yield abstract pseudobases.  Note the following extends \cite[Theorem 8.4]{BiceStarling2016} to non-0-dimensional spaces.

\begin{thm}\label{TopologyRecovery}
If $P\not\ni\emptyset$ is a pseudobasis of $X$ and we take $\prec\ =\ \Subset$ then
\begin{align}
\label{QperpR}Q\perp R\qquad&\Leftrightarrow\qquad(\bigcup Q)\cap(\bigcup R)=\emptyset.\\
\label{QDR}Q\D R\qquad&\Leftrightarrow\qquad\bigcup Q\subseteq\overline{\bigcup R}.\\
\intertext{If $X$ is also Hausdorff then $P$ is an abstract pseudobasis with}
\label{QCR}Q\C R\qquad&\Leftrightarrow\qquad\bigcup Q\Subset\bigcup R
\end{align}
and $x\mapsto T_x=\{O\in P:x\in O\}$ is a homeomorphism from $X$ onto $\mathcal{T}(P)$.
\end{thm}

\begin{proof}\
\begin{itemize}
\item[\eqref{QperpR}]  If $(\bigcup Q)\cap(\bigcup R)=\emptyset$ then certainly $Q^\succ\cap R^\succ\subseteq Q^\supseteq\cap R^\supseteq=\emptyset$.  If $(\bigcup Q)\cap(\bigcup R)\neq\emptyset$ then we must have $O\cap N\neq\emptyset$, for some $O\in Q$ and $N\in R$.  Then \eqref{Dense} and \eqref{PointRound} yields $M\in P$ with $M\Subset O\cap N$ so $O\not\perp N$.

\item[\eqref{QDR}]  If $Q\not\D R$ then we have $O\in Q^\succ$ with $R\perp O$.  Thus $O\subseteq\bigcup Q$ and $O\cap\bigcup R=\emptyset=O\cap\overline{\bigcup R}$ so $\bigcup Q\not\subseteq\overline{\bigcup R}$.  Conversely, if $\bigcup Q\not\subseteq\overline{\bigcup R}$ then $N\setminus\overline{\bigcup R}\neq\emptyset$, for some $N\in Q$.  Then \eqref{Dense} and \eqref{PointRound} yields $O\in P$ with $O\Subset N\setminus\bigcup R$ so $Q\not\D R$.

\item[\eqref{QCR}]  If $G$ is Hausdorff and $\bigcup Q\Subset\bigcup R$ then $\overline{\bigcup Q}$ is compact and $\overline{\bigcup Q}\subseteq\bigcup R$.  Thus, for each $g\in\overline{\bigcup Q}$, we have $O\in R$ with $g\in O$.  Then \eqref{PointRound} yields $N\in P$ with $g\in N\Subset O$.  As $\overline{\bigcup Q}$ is compact we can therefore cover it with finite $F\subseteq O^\succ\subseteq R^\succ$ so $Q\D F\prec R$, i.e. $Q\C R$.  Conversely, if $Q\D F\prec R$ then $\overline{\bigcup{Q}}\subseteq\overline{\bigcup{F}}=\bigcup_{O\in F}\overline{O}\subseteq\bigcup R$, as $F$ is finite.  Also $\bigcup_{O\in F}\overline{O}$ is compact, as a finite union of compact sets, and hence $\bigcup Q\Subset\bigcup R$.
\end{itemize}

Certainly $\Subset$ is transitive and $P$ is round, by \eqref{PointRound}.  Moreover, as above we see that whenever $O\Subset N$ we can find $F$ with $O\C F\prec N$, so $P$ also satisfies \eqref{Shrinking} and is thus an abstract pseudobasis.

Now each $T_x$ is non-empty and round by \eqref{Cover} and \eqref{PointRound}.  We claim each $T_x$ is also tight.  Indeed, for any finite $F\subseteq T_x$ and $H\prec P\setminus T_x$, note that $O\in H$ means $O\Subset N\not\ni x$, for some $N\in P$, and hence $x\notin\overline{O}$, as $X$ is Hausdorff.  Thus $x\in\bigcap F\setminus\overline{\bigcup H}$ so we have some $O\in P$ with $O\Subset\bigcap F\setminus\overline{\bigcup H}$, by \eqref{Dense}, thus verifying \eqref{Tight-Round}.

Thus $x\mapsto T_x$ maps $X$ to $\mathcal{T}(P)$.  By \eqref{Separating}, $x\mapsto T_x$ is injective.  To see that it is also surjective, take $T\in\mathcal{T}(P)$ and note that
\[C=\bigcap T\setminus\bigcup(P\setminus T)=\bigcap_{\substack{O\in T\\ N\prec P\setminus T}}\overline{O}\setminus N,\]
as $T$ is round and $P$ is point-round (so that we can replace $N\in P\setminus T$ with $N\prec P\setminus T$ at the end).  As $T$ is tight, this last collection of sets has the finite intersection property.  As it also consists of compact closed subsets, the entire intersection must also be non-empty and hence $x\in C$, for some $x\in G$.  By the definition of $C$, if $O\in T$ then $x\in\bigcap T\subseteq O$ so $O\in T_x$, while if $O\in P\setminus T$ then $x\notin\bigcup P\setminus T\supseteq O$ so $O\notin T_x$.  Thus $T=T_x$, showing that $x\mapsto T_x$ is indeed surjective.

Then we immediately see that $x\mapsto T_x$ is a homeomorphism with respect to the topologies on $X$ and $\mathcal{T}(P)$ generated by $P$ and $(O_p)_{p\in P}$ respectively.  Thus $x\mapsto T_x$ is also a homeomorphism with respect to the corresponding patch topologies, which are just the original topologies of $X$ and $\mathcal{T}(P)$, by \autoref{PseudoPatch} and \autoref{abstract->concrete}.
\end{proof}

In summary, a concrete pseudobasis of a locally compact Hausdorff space always yields an abstract pseudobasis.  Moreover, if we represent this abstract pseudobasis as a concrete pseudobasis of its tight spectrum, we recover the original pseudobasis.

However, if we represent a given abstract pseudobasis $(P,\prec)$ as a concrete pseudobasis $((O_p)_{p\in P},\Subset)$ then we lose $\prec$, as $\Subset$ corresponds instead to the weaker relation $\C$ on $P$.  We could even have $O_p=O_q$ for distinct $p,q\in P$.  If we want a true duality then we should really restrict to the abstract pseudobases with $\prec\ =\ \C$ on $P$.  More explicitly, these are the \emph{separative} pseudobases, i.e. satisfying
\[\label{Separative}\tag{Separative}p\not\prec q\qquad\Rightarrow\qquad\forall\text{ finite }F\prec q\ \exists r\prec p\ (r\perp F).\]
Thus the duality we obtain can be represented as
\[\text{separative abstract pseudobases}\qquad\leftrightarrow\qquad\text{pseudobases of Hausdorff spaces}.\]

\begin{rmk}
When $\prec$ is reflexive it suffices to consider $F=q$ in \eqref{Separative}, in which case this reduces to the usual notion of a separative poset (see \cite[Definition III.4.10]{Kunen2011}) and $(P,\C)$ is the separative quotient of $(P,\prec)$.  In general, if $(P,\prec)$ is an abstract pseudobasis then $(P,\C)$ is again an abstract pseudobasis, which can be seen by going through the representation above or more directly using the basic properties in \autoref{DCproperties}.
\end{rmk}

\section{Inverse Semigroups}\label{InverseSemigroups}

Next we examine what happens when our abstract pseudobases have some extra inverse semigroup structure.  In this case, each tight subset will be a coset which will naturally give the tight spectrum an extra \'etale groupoid structure.  First, however, we need some general results about inverse semigroups and their cosets.

Recall that an \emph{inverse semigroup} (see \cite{Lawson1998}) is a semigroup $S$ in which every $s\in S$ has a unique generalised inverse $s^{-1}$, i.e. satisfying
\[ss^{-1}s=s\qquad\text{and}\qquad s^{-1}ss^{-1}=s^{-1}.\]
Actually, uniqueness here is equivalently to saying that the idempotents
\[E=\{s\in S:s=s^2\}=\{s^{-1}s:s\in S\}=\{ss^{-1}:s\in S\}\]
always commute.  The \emph{canonical partial order} $\leq$ on $S$ is given by
\[s\leq t\qquad\Leftrightarrow\qquad s^{-1}s=s^{-1}t\qquad\Leftrightarrow\qquad ss^{-1}=ts^{-1}\qquad\Leftrightarrow\qquad s\in Et=tE.\]
Also, a \emph{zero} in $S$ is a (necessarily unique) element $0\in S$ such that
\[0S=S0=\{0\}.\]
Note this implies $0\in E$ and $S\geq 0$.

\subsection{Cosets}

\begin{center}
\textbf{Throughout this section we assume $S$ is an inverse semigroup.}
\end{center}

\begin{dfn}[{\cite[\S3.1]{LawsonMargolisSteinberg2013}}]
We call $C\subseteq S$ a \emph{coset} if $CC^{-1}C=C^\leq=C$.
\end{dfn}

As $s=ss^{-1}s$, for all $s\in S$, for $C$ to be a coset it suffices that
\[\tag{Coset}CC^{-1}C\subseteq C^\leq\subseteq C.\]
Moreover, even if $C$ just satisfies $CC^{-1}C\subseteq C^\leq$, we can still obtain a coset by taking the upwards closure, as $C^\leq C^{\leq-1}C^\leq\subseteq(CC^{-1}C)^\leq\subseteq C^{\leq\leq}=C^\leq$.

\begin{prp}
For any non-empty $T,U,C\subseteq S$, if $C$ is a coset then
\begin{equation}\label{T-1U}
T^{-1}U\subseteq(CC^{-1})^\leq\qquad\Leftrightarrow\qquad T^{-1}UC\subseteq C\qquad\Rightarrow\qquad(TC)^\leq=(UC)^\leq.
\end{equation}
\end{prp}

\begin{proof}
If $T^{-1}U\subseteq(CC^{-1})^\leq$ then
\[T^{-1}UC\subseteq(CC^{-1})^\leq C\subseteq(CC^{-1}C)^\leq=C^\leq=C.\]
Conversely, if $T^{-1}UC\subseteq C(\neq\emptyset$ and $T\neq\emptyset$) then
\begin{align*}
T^{-1}U&\subseteq(T^{-1}UCC^{-1})^\leq\subseteq(CC^{-1})^\leq\quad\text{and}\\
(UC)^\leq&\subseteq(TT^{-1}UC)^\leq\subseteq(TC)^\leq.
\end{align*}
But $T^{-1}U\subseteq(CC^{-1})^\leq$ implies $U^{-1}T=(T^{-1}U)^{-1}\subseteq(CC^{-1})^{\leq-1}=(CC^{-1})^\leq$ and hence $U^{-1}TC\subseteq C$ which (as $U\neq\emptyset$) likewise yields
\[(TC)^{-1}\subseteq(UU^{-1}TC)^\leq\subseteq(UC)^\leq\]
and hence $(TC)^\leq=(UC)^\leq$.
\end{proof}

\begin{prp}\label{TCcoset}
If $T^{-1}T\subseteq(CC^{-1})^\leq$ and $C$ is coset then so is $(TC)^\leq$.
\end{prp}

\begin{proof}
By \eqref{T-1U}, $T^{-1}T\subseteq(CC^{-1})^\leq$ implies $T^{-1}TC\subseteq C$ so
\[TC(TC)^{-1}TC=TCC^{-1}T^{-1}TC\subseteq TCC^{-1}C=TC.\]
Thus the same applies to $(TC)^\leq$ and hence $(TC)^\leq$ is a coset.
\end{proof}

\begin{prp}\label{IdempotentCosets}
If $C$ is a coset then $C\cap E\neq\emptyset$ iff $C=(CC^{-1})^\leq$.
\end{prp}

\begin{proof}
If $e\in C\cap E$ then $C\subseteq(Ce)^\leq\subseteq(CC^{-1})^\leq$ and
\[(CC^{-1})^\leq\subseteq((CC^{-1})^\leq e)^\leq\subseteq(CC^{-1}C)^\leq=C.\]
On the other hand, if $C=(CC^{-1})^\leq$ then $cc^{-1}\in C\cap E$, for any $c\in C$.
\end{proof}

In \cite[\S3.1]{LawsonMargolisSteinberg2013} the cosets are considered as an inverse semigroup with the product $C\otimes D$ defined to be the coset generated by $CD$.  In particular, we always have $(CD)^\leq\subseteq C\otimes D$.  In the cases we are interested in, equality will even hold here.  Indeed, by \autoref{TCcoset} with $T=C^{-1}$, we always have
\[C^{-1}\otimes C=(C^{-1}C)^\leq.\]
Thus $C^{-1}\otimes C=D\otimes D^{-1}$ is equivalent to $(C^{-1}C)^\leq=(DD^{-1})^\leq$, in which case we again have $C\otimes D=(CD)^\leq$.  Indeed, for any $c,c'\in C'$, we have $c^{-1}c'\in(DD^{-1})^\leq$ so $(cD)^\leq=(c'D)^\leq$, by \eqref{T-1U}, and hence $(CD)^\leq=(cD)^\leq$ is a coset, again by \autoref{TCcoset}.

As with any inverse semigroup, this restricted product yields a \emph{groupoid} (see \cite[\S3.1]{Lawson1998}), i.e. a set $G$ with an inverse operation and partial product operation making $G$ a small category in which all morphisms are isomorphisms.

\begin{dfn}
The \emph{coset groupoid} of $S$ is the set of cosets of $S$ with operations
\begin{align}
\tag{Inverse}C&\mapsto C^{-1}\quad\text{and}\\
\tag{Product}(C,D)&\mapsto(CD)^\leq\quad\text{iff}\quad(C^{-1}C)^\leq=(DD^{-1})^\leq.
\end{align}
\end{dfn}

\subsection{Ordered Inverse Semigroups}

We will be interested inverse semigroups with an extra (possibly non-reflexive) order relation $\prec$ which behaves well with respect to the inverse semigroup structure.

\begin{dfn}
An \emph{ordered inverse semigroup} is an inverse semigroup $S$ together with a transitive left auxiliary relation $\prec\ \subseteq\ \leq$ preserving products and inverses, i.e.
\begin{align}
\label{LeftAuxiliary}\tag{Left Auxiliary}q\prec q'\leq r\qquad&\Rightarrow\qquad q\prec r\qquad\Rightarrow\qquad q\leq r.\\
\label{Multiplicative}\tag{Multiplicative}q\prec q'\quad\text{and}\quad r\prec r'\qquad&\Rightarrow\qquad qr\prec q'r'.\\
\label{Inverses}\tag{Invertive}q\prec r\qquad&\Rightarrow\qquad q^{-1}\prec r^{-1}.
\end{align}
\end{dfn}

\begin{rmk}\label{CanonicalOrderedInverseSemigroup}
Any inverse semigroup $S$ becomes an ordered inverse semigroup when we take $\prec\ =\ \leq$.  The point is that there may be many other strictly stronger relations $\prec$ which make $S$ an ordered inverse semigroup.
\end{rmk}

One immediate consequence is the following auxiliary-product property.

\begin{prp}
If $r^{-1}r\leq ss^{-1}$ and $s\prec S$ in an ordered inverse semigroup $S$,
\begin{equation}\label{ac<bc}
q\prec r\qquad\Rightarrow\qquad qs\prec rs.
\end{equation}
\end{prp}

\begin{proof}
Take $t\in S$ with $s\prec t$ and hence $s\leq t$.  By \eqref{Multiplicative},
\[qs\prec rt=rr^{-1}rt=rr^{-1}rss^{-1}t=rs.\qedhere\]
\end{proof}

\begin{prp}\label{Frink=>Coset}
Every Frink filter $T$ in an ordered inverse semigroup is a coset.
\end{prp}

\begin{proof}
To see that $T=T^\leq$, note that if $u\geq t\in T$ then, as $T$ is a Frink filter, we have $w\in P$ and finite $F\subseteq T$ with $\widehat{F}\prec w\prec t\leq u$.  As $\prec$ is auxiliary to $\leq$, this yields $\widehat{F}\prec w\prec u$ and hence $u\in T$, as $T$ is a Frink filter.

To see that $T=TT^{-1}T$, again note that, for any $a,b,c\in T$, we have $a',b',c'\in P$ and finite $F,G,H\subseteq T$ with $\widehat{F}\prec a'\prec a$, $\widehat{G}\prec b'\prec b$ and $\widehat{H}\prec c'\prec c$.  For any $d\in\widehat{I}$, where $I=F\cup G\cup H$, \eqref{Multiplicative} and \eqref{Inverses} yields
\[d=dd^{-1}d\prec a'b'^{-1}c'\prec ab^{-1}c,\]
i.e. $\widehat{I}\prec a'b'^{-1}c'\prec ab^{-1}c$.  Thus $ab^{-1}c\in T$, as $T$ is a Frink filter.
\end{proof}

\begin{center}
\textbf{For the rest of this section we assume that $(S,\prec)$ is an ordered inverse semigroup with zero and $P=S\setminus\{0\}$ is round}.
\end{center}

\begin{rmk}
If $\prec$ is not just left but also right auxiliary, i.e.
\[\label{RightAuxiliarity}\tag{Right Auxiliary}q\leq r'\prec r\qquad\Rightarrow\qquad q\prec r,\]
then, as $P$ is round, for any $p\in P$, we have $p'\in P$ with $0\leq p'\prec p$ and hence $0\prec p$.  As long as we also have some $p,q\in P$ with $pq=0$ then \eqref{Multiplicative} yields $0=00\prec pq=0$ too.  So in this case $0$ is not just a $\leq$-minimum but also a $\prec$-minimum (which is moreover unique, as $\prec\ \subseteq\ \leq$).
\end{rmk}

We define $\D$, $\C$ and $\perp$ on $P$ from $\prec$ exactly as before.  For convenience, we extend these relations to $S$ by simply ignoring $0$, e.g. for $Q,R\subseteq S$,
\[Q\D R\qquad\Leftrightarrow\qquad(Q\setminus\{0\})\D(R\setminus\{0\}).\]

First we note that $\perp$ and $\D$ can equivalently be defined from $\leq$ rather than $\prec$.

\begin{prp}
For all $Q,R\subseteq S$ and $s\in S$,
\begin{align}
\label{QperpR0}Q^\geq\cap R^\geq=\{0\}\qquad\Leftrightarrow\qquad&Q\perp R\qquad\Rightarrow\qquad Qs\perp Rs.\\
\label{QDR0}Q^\geq\setminus\{0\}\ \geq\ R^\geq\setminus\{0\}\qquad\Leftrightarrow\qquad&Q\D R\qquad\Rightarrow\qquad Qs\D Rs.
\end{align}
\end{prp}

\begin{proof}\
\begin{itemize}
\item[\eqref{QperpR0}]  If $q,r\geq p\neq0$ then, as $P$ is round, we have $p'\neq0$ with $q,r\geq p\succ p'$ and hence $q,r\succ p'$, by \eqref{LeftAuxiliary}, so $q\not\perp r$.  Conversely, if $q\not\perp r$ then we have $p\neq0$ with $q,r\succ p$ and hence $q,r\geq p$, as $\prec\ \subseteq\ \leq$.  This proves the $\Leftrightarrow$ part.  For $\Rightarrow$, just note $q\wedge r=0$ implies $qs\wedge rs=(q\wedge r)s=0$, by \cite[Proposition 1.4.19]{Lawson1998}.

\item[\eqref{QDR0}]  If $Q\D R$, i.e. $Q^\succ\setminus\{0\}\succ R^\succ\setminus\{0\}$, then
\[Q^\geq\setminus\{0\}\ \succ\ Q^\succ\setminus\{0\}\ \succ\ R^\succ\setminus\{0\}\ \subseteq\ R^\geq\setminus\{0\}\]
and hence $Q^\geq\setminus\{0\}\geq R^\geq\setminus\{0\}$, as $P$ is round and \eqref{LeftAuxiliary} holds.  Conversely, if $Q^\geq\setminus\{0\}\geq R^\geq\setminus\{0\}$ then likewise
\[Q^\succ\setminus\{0\}\ \subseteq\ Q^\geq\setminus\{0\}\ \geq\ R^\geq\setminus\{0\}\ \succ\ R^\succ\setminus\{0\}\]
and hence $Q^\succ\setminus\{0\}\succ R^\succ\setminus\{0\}$.

Now assume $Q\D R$ and $0\neq p\leq qs$, for some $q\in Q$.  Then $0\neq p=pp^{-1}p\leq qsp^{-1}p$ so $sp^{-1}\neq0$.  Thus $0\neq ps^{-1}\leq qss^{-1}\leq q$.  As $Q\D R$, we have non-zero $r\leq ps^{-1}$ with $r\leq R$.  Then again $0\neq rs\leq ps^{-1}s\leq p$, showing that $Qs\D Rs$.\qedhere
\end{itemize}
\end{proof}

\begin{rmk}
If $S$ is separative, i.e. if $\C$ and $\prec$ coincide on singletons, then \eqref{RightAuxiliarity} holds.  Indeed \eqref{DC=>C}, \eqref{FprecQ}, \eqref{QDR0} and separativity yield
\[q\leq r'\prec r\quad\Rightarrow\quad q\D r'\C r\quad\Rightarrow\quad q\C r\quad\Rightarrow\quad q\prec r.\]
\end{rmk}

\begin{cor}
For any $Q,Q',R,R'\subseteq S$,
\begin{align}
\label{QRDQ'R'}Q\D Q'\quad\text{and}\quad R\D R'\qquad&\Rightarrow\qquad QR\D Q'R'.\\
\label{QRCQ'R'}Q\C Q'\quad\text{and}\quad R\C R'\qquad&\Rightarrow\qquad QR\C Q'R'.
\end{align}
\end{cor}

\begin{proof}\
\begin{itemize}
\item[\eqref{QRDQ'R'}]  If $Q\D Q'$ and $R\D R'$ then $QR\D Q'R\D Q'R'$, by \eqref{QDR0} and \eqref{Dcup} (or rather an extension of \eqref{Dcup} to infinite unions).
\item[\eqref{QRCQ'R'}]  If $Q\C Q'$ and $R\C R'$ then $Q\D F\prec Q'$ and $R\D G\prec R'$, for some finite $F,G\subseteq S$, and hence $QR\D FG\prec Q'R'$, by \eqref{QRDQ'R'} and \eqref{Multiplicative}.\qedhere
\end{itemize}
\end{proof}

\begin{prp}
If $s\prec S$ and $ss^{-1}\in(T^{-1}T)^\leq$ then
\begin{equation}\label{TsTight}
T\in\mathcal{T}(P)\qquad\Rightarrow\qquad(Ts)^\leq\in\mathcal{T}(P).
\end{equation}
\end{prp}

\begin{proof}
If $T$ is tight then $T$ is a round Frink filter and hence a coset, by \autoref{Frink=>Coset}.  Also, for any $t\in T$, \eqref{T-1U} yields $tss^{-1}\in Tss^{-1}\subseteq T$.  As $T$ is round, we have $u\in T$ with $u\prec tss^{-1}$ and then \eqref{ac<bc} yields $us\prec tss^{-1}s=ts$.  This shows that $Ts$ is round and hence so is $(Ts)^\leq$, by \eqref{LeftAuxiliary}.

Now we need to show that $(Ts)^\leq\not\CC S\setminus(Ts)^\leq$.  For this it suffices to show that
\[Ts\CC H\qquad\Rightarrow\qquad H\cap(Ts)^\leq\neq\emptyset.\]
So take finite $F\subseteq Ts$ with $\widehat{F}\C H$.  Then \eqref{ac<bc} yields $\widehat{Fs^{-1}}=\widehat{F}s^{-1}\C Hr^{-1}$, for any $r\succ s$, by \eqref{FprecQ} and \eqref{QRCQ'R'}.  As $T$ is tight and $Fs^{-1}\subseteq Tss^{-1}\subseteq T$, we have $h\in H$ with $hr^{-1}\in T$.  Thus $h\geq hr^{-1}r\geq hr^{-1}s\in Ts$ so $H\cap(Ts)^\leq\neq\emptyset$.
\end{proof}

\begin{cor}\label{TightGroupoid}
$\mathcal{T}(P)$ is a subgroupoid of the coset groupoid of $S$.
\end{cor}

\begin{proof}
Take any $T,U\in\mathcal{T}(P)$ with $(T^{-1}T)^\leq=(UU^{-1})^\leq$.  As $U$ is round, we have $u\in U$ with $u\prec U\subseteq S$.  Thus $(TU)^\leq=(Tu)^\leq\in\mathcal{T}(P)$, by \eqref{T-1U} and \eqref{TsTight}.  Thus $\mathcal{T}(P)$ is closed under the restricted product.  By \eqref{Inverses}, $\mathcal{T}(P)$ is also immediately seen to be closed under taking inverses.
\end{proof}

\begin{dfn}
We call $\mathcal{T}(P)$ the \emph{tight groupoid} of $S$.
\end{dfn}

\subsection{\'Etale Groupoids}

Recall that an \emph{\'etale groupoid} $G$ is a groupoid together with a topology that behaves well with respect to the groupoid operations, specifically
\begin{gather*}
g\mapsto g^{-1}g\quad\text{is a local homeomorphism and}\\
(g,h)\mapsto gh\quad\text{and}\quad g\mapsto g^{-1}\quad\text{are continuous}.
\end{gather*}

However, we will use an equivalent more algebraic definition.  First let
\begin{align*}
G^{(0)}&=\{g^{-1}g:g\in G\}.\\
G^{(2)}&=\{(g,h):g^{-1}g=hh^{-1}\}.\\
\mathcal{B}(G)&=\{B\subseteq G:B^{-1}B,BB^{-1}\subseteq G^{(0)}\}.\\
\mathcal{OB}(G)&=\mathcal{O}(G)\cap\mathcal{B}(G).
\end{align*}
So $G^{(0)}$ is the unit space of $G$, $G^{(2)}$ is the set of composable pairs, $\mathcal{B}(G)$ denotes the bisections of $G$ and $\mathcal{OB}(G)$ denotes the open bisections of $G$.

\begin{dfn}
We call a collection of bisections $\mathcal{B}\subseteq\mathcal{B}(G)$ an \emph{\'etale family} if $\mathcal{B}$ covers $G$ and is closed under pointwise products and inverses, i.e.
\[\tag{\'Etale Family}G=\bigcup\mathcal{B}\quad\text{and}\quad(O,N\in\mathcal{B}\quad\Rightarrow\quad ON\in\mathcal{B}\text{ and }O^{-1}\in\mathcal{B}).\]
An \emph{\'etale (pseudo)(sub)basis} is a (pseudo)(sub)basis that is also an \'etale family.
\end{dfn}

\begin{prp}\label{EtaleSubbasis}
$G$ is an \'etale groupoid if and only if $G$ has an \'etale subbasis.
\end{prp}

\begin{proof}
By \cite[Proposition 6.6]{BiceStarling2018}, $G$ is \'etale if and only if $G$ has an \'etale basis.  To show that an \'etale subbasis suffices, one only has to note that the pointwise product preserves intersections of bisections, i.e. for any $\mathcal{A},\mathcal{B}\subseteq\mathcal{B}(G)$.
\begin{equation}\label{IntersectionPreserving}
(\bigcap\mathcal{A})(\bigcap\mathcal{B})=\bigcap\mathcal{A}\mathcal{B}(=\bigcap_{\substack{A\in\mathcal{A}\\ B\in\mathcal{B}}}AB).
\end{equation}
This follows from the fact that, for any $B\in\mathcal{B}(G)$ and $\mathcal{A}\subseteq\mathcal{P}(G)$,
\[B\bigcap\mathcal{A}=\bigcap B\mathcal{A}(=\bigcap_{A\in\mathcal{A}}BA).\]
Indeed, $B\bigcap\mathcal{A}\subseteq\bigcap B\mathcal{A}$ is immediate, while $B^{-1}B\subseteq G^{(0)}$ yields
\[\bigcap B\mathcal{A}\subseteq BB^{-1}\bigcap B\mathcal{A}\subseteq B\bigcap B^{-1}B\mathcal{A}\subseteq B\bigcap\mathcal{A}.\qedhere\]
\end{proof}

In Hausdorff groupoids even \'etale pseudosubbases suffice, as we now show.

\begin{prp}\label{CoherentEtale}
If $G$ is a coherent \'etale groupoid then $G^\mathrm{patch}$ is also \'etale.
\end{prp}

\begin{proof}
First note saturated subsets are closed under taking products (even in an arbitrary \'etale groupoid).  Indeed, as $\mathcal{OB}(G)$ covers $G$, for any $Q,R\subseteq\mathcal{O}(G)$,
\begin{align*}
\big(\bigcap Q\big)\big(\bigcap R\big)&=\big(\bigcup_{A\in\mathcal{OB}(G)}A\cap\bigcap Q\big)\big(\bigcup_{B\in\mathcal{OB}(G)}B\cap\bigcap R\big)\\
&=\big(\bigcup_{A\in\mathcal{OB}(G)}\bigcap_{C\in Q}A\cap C\big)\big(\bigcup_{B\in\mathcal{OB}(G)}\bigcap_{D\in R}B\cap D\big)\\
&=\bigcup_{A,B\in\mathcal{OB}(G)}\big(\bigcap_{C\in Q}A\cap C\big)\big(\bigcap_{D\in R}B\cap D\big)\\
&=\bigcup_{A,B\in\mathcal{OB}(G)}\bigcap_{\substack{C\in Q\\ D\in R}}(A\cap C)(B\cap D),
\end{align*}
by \eqref{IntersectionPreserving}.  As $G$ is \'etale, each $(A\cap C)(B\cap D)$ is open and hence each intersection $\bigcap_{C\in Q,D\in R}(A\cap C)(B\cap D)$ is saturated.  As open sets are closed under arbitrary unions, so are saturated subsets and hence the product is indeed saturated.

Next we claim that products of compact saturated subsets are again compact.  To see this, take $C,D\in\mathcal{C}(G)$.  For each $c\in C$, we have $O,N\in\mathcal{OB}(G)$ with $c\in O\Subset N$, as $\mathcal{OB}(G)$ covers $G$ and $G$ is locally compact.  As $C$ is compact, we can cover it with finitely many such $O$.  This yields a finite $F\subseteq\mathcal{C}(G)\cap\mathcal{OB}(G)^\Supset$ with $C\subseteq\bigcup F$, i.e. each $C'\in F$ is compact, saturated and contained in some open bisection.  As $G$ is coherent, $C'\cap C\in\mathcal{C}(G)$, for all $C'\in F$, and $C=\bigcup_{C'\in F}C'\cap C$.  The same applies to $D$ and thus $CD$ can be expressed as a finite union of products of compact saturated subsets, each of which is contained in some open bisection.  As $\mathcal{C}(G)$ is closed under finite unions, it suffices to consider the case that $C$ and $D$ are already contained in open bisections $O$ and $N$ respectively.

Now consider the maps $s(g)=g^{-1}g$ and $r(g)=gg^{-1}$ restricted to $O$ and $N$ respectively, which are homeomorphisms.  In particular, as they are continuous, $C^{-1}C$ and $DD^{-1}$ are compact and also saturated, as mentioned above.  As $G$ is coherent, the same applies to their intersection, i.e. $E=C^{-1}C\cap DD^{-1}\in\mathcal{C}(G)$.  Now $p(g)=s^{-1}(g)r^{-1}(g)$ restricted to $O^{-1}O\cap NN^{-1}$ is also a homeomorphism and hence $CD=p[E]$ is also compact and saturated, i.e. $CD\in\mathcal{C}(G)$, as required.

Lastly, note that $G^\mathrm{patch}$ has a basis of elements of the form $O\setminus C$ where $O\in\mathcal{O}(G)\cap\mathcal{OB}(G)^\Supset$ and $C\in\mathcal{C}(G)$, so each $O$ is open and contained in some compact bisection $B_O$.  But then, for any such $O\setminus C$ and $N\setminus D$, 
\[(O\setminus C)(N\setminus D)=ON\setminus(OD\cup CN)=ON\setminus(B_OD\cup CB_N).\]
As $\mathcal{C}(G)$ is closed under products, this shows that our basis is closed under products too, and it is immediately seen to be closed under inverses as well.
\end{proof}

\begin{cor}\label{EtalePseudosubbasis}
If $G$ is Hausdorff with an \'etale pseudosubbasis $P$ then $G$ is \'etale.
\end{cor}

\begin{proof}
By \autoref{EtaleSubbasis}, $G_P(=$ $G$ with the topology generated by $P)$ is \'etale.  By \autoref{PseudoPatch} and \autoref{CoherentEtale}, $G=G_P^\mathrm{patch}$ is \'etale.
\end{proof}

Later we will need the following elementary result on \'etale families.

\begin{prp}
If $\mathcal{B}\subseteq\mathcal{B}(G)$ is \'etale, $T_g=\{O\in\mathcal{B}:g\in O\}$ and $\leq\ =\ \subseteq$ on $\mathcal{B}$,
\begin{equation}\label{Functor}
T_{g^{-1}}=T_g^{-1}\qquad\text{and}\qquad T_{gh}=(T_gT_h)^\leq,\quad\text{for all }(g,h)\in G^{(2)}.
\end{equation}
\end{prp}

\begin{proof}
We immediately see that $T_g^{-1}=T_{g^{-1}}$ and $T_gT_h\subseteq T_{gh}$ and hence $(T_gT_h)^\leq\subseteq T_{gh}$, as $gh\in O\subseteq N$ implies $gh\in N$.  Conversely, take any $O\in T_g$ and note that, as $O$ is a bisection, $OO^{-1}N\subseteq N$, for any $N\in T_{gh}$, so
\[T_{gh}\subseteq(OO^{-1}T_{gh})^\leq\subseteq(T_gT_{g^{-1}}T_{gh})^\leq\subseteq(T_gT_{g^{-1}gh})^\leq\subseteq(T_gT_h)^\leq.\qedhere\]
\end{proof}

\subsection{Pseudobasic Inverse Semigroups}

Now we can finally show that, for suitable inverse semigroups, our tight groupoid is necessarily \'etale.

\begin{dfn}
We call $(S,\prec)$ a \emph{pseudobasic inverse semigroup} if $S$ is an ordered inverse semigroup with zero and $P=S\setminus\{0\}$ is an abstract pseudobasis.
\end{dfn}

\begin{thm}\label{LCHEgroupoid}
If $S$ is a pseudobasic inverse semigroup then $\mathcal{T}(P)$ is a locally compact Hausdorff \'etale groupoid with an \'etale pseudobasis $(O_s)_{s\prec S}$.
\end{thm}

\begin{proof}
By \autoref{TightHausdorff}, \autoref{TightLocallyCompact} and \autoref{TightGroupoid}, $\mathcal{T}(P)$ is locally compact Hausdorff groupoid.  By \autoref{abstract->concrete}, $(O_p)_{p\in P}$ and hence $(O_s)_{s\prec S}$ is a pseudo(sub)basis for $\mathcal{T}(P)$.  In particular, $(O_s)_{s\prec S}$ covers $\mathcal{T}(P)$ so, by \autoref{EtalePseudosubbasis}, it suffices to verify that $(O_s)_{s\prec S}$ is an \'etale family.

First note that, for any $s\prec S$,
\begin{equation}\label{ss-1}
O_sO_s^{-1}=O_{ss^{-1}}.
\end{equation}
Indeed, $O_sO_s^{-1}\subseteq O_{ss^{-1}}$ is immediate.  Conversely, if $T\in O_{ss^{-1}}$ then $U=(Ts)^\leq\in\mathcal{T}(P)$, by \eqref{TsTight}.  Thus $U\in O_s$, $U^{-1}\in O_{s^{-1}}$ and $(UU^{-1})^\leq=(Tss^{-1}T^{-1})^\leq=(Tss^{-1}ss^{-1})^\leq=T$ so $T\in O_sO_{s^{-1}}$, as required.  As $ss^{-1}\in E$, \autoref{IdempotentCosets} yields $O_{ss^{-1}}\in\mathcal{T}(P)^{(0)}$.  Likewise, $O_s^{-1}O_s=O_{s^{-1}s}\in\mathcal{T}(P)^{(0)}$ so $O_s\in\mathcal{OB}(\mathcal{T}(P))$.

We also immediately have $O_{s}^{-1}=O_{s^{-1}}$ and $O_sO_t\subseteq O_{st}$, for any $s,t\in S$.  Conversely, as long as $s\prec S$, \eqref{ss-1} yields
\[O_{st}=O_{st}O_{t^{-1}s^{-1}}O_{st}\subseteq O_{stt^{-1}s^{-1}}O_{st}\subseteq O_{ss^{-1}}O_{st}=O_sO_{s^{-1}}O_{st}\subseteq O_sO_{s^{-1}st}\subseteq O_sO_t.\]
Thus $(O_s)_{s\prec S}$ is also closed under products, as required.
\end{proof}

Exel's original tight groupoid construction was obtained from the groupoid of germs coming from the canonical partial action of $S$ on the tight spectrum of $E$.  In the Hausdorff case, this agrees with our tight groupoid when we consider $S$ as a pseudobasic inverse semigroup with $\prec\ =\ \leq$ (see \autoref{InterpolationRemark} and \autoref{CanonicalOrderedInverseSemigroup}).  Indeed, it follows from the Hausdorff characterization in \cite[Theorem 3.16]{ExelPardo2016} that in this case every one of our tight subsets will be a true filter, not just a Frink filter, and these can be identified with germs just like in Lenz's alternative construction of Paterson's universal groupoid (see \cite{Lenz2008} and \cite[\S3.3]{LawsonMargolisSteinberg2013}).  In the non-Hausdorff case, our tight groupoid will instead be a `Hausdorffification' of Exel's original tight groupoid.  To illustrate this, we give a simple example.

\begin{xpl}
Let $G=\mathbb{N}\cup\{\infty,\infty'\}$ be the two-point compactification of $\mathbb{N}$, so the neighborhoods of $\infty$ are of the form $C\cup\{\infty\}$, for some cofinite $C\subseteq\mathbb{N}$, and likewise for the neighborhoods of $\infty'$.  We make $G$ an ample \'etale groupoid by letting the unit space be $G^0=\mathbb{N}\cup\{\infty\}$ and defining $\infty'\cdot\infty'=\infty$.  Let
\[S=\{\{n\}:n\in\mathbb{N}\}\cup\{e\}\cup\{s\}\cup\{\emptyset\},\]
where $e=\mathbb{N}\cup\{\infty\}$ and $s=\mathbb{N}\cup\{\infty'\}$.  So $S$ is a pseudobasis of compact open bisections and an inverse semigroup under pointwise operations with all elements idempotent except $s$.  If we applied Exel's original tight groupoid construction to $S$, we would recover $G$.  However, our tight groupoid construction identifies $\infty$ and $\infty'$ because every $T\in\mathcal{T}(S\setminus\{\emptyset\})$ must contain both $s$ and $e$, as $s^\perp=e^\perp=\{\emptyset\}$.
\end{xpl}

\begin{rmk}
If we want a faithful extension of Exel's original contruction, even in the non-Hausdorff case, then this can still be done in a similar way.  Basically, the idea is to modify the definition of (both abstract and concrete) pseudobases and work with `local' versions of tight subsets and the tight spectrum, as we discuss in a follow up paper.  However, the Hausdorff construction presented here may be of some interest in its own right, even in the original $\prec\ =\ \leq$ case.  Indeed, it would be interesting to see if our Hausdorff tight groupoid might be a suitable alternative to some of the non-Hausdorff tight groupoids already considered in the literature.
\end{rmk}

We can also turn our tight groupoid construction into a duality:
\begin{gather*}
\text{separative pseudobasic inverse semigroups}\\
\updownarrow\\
\text{\'etale pseudobases of Hausdorff groupoids}.
\end{gather*}
Indeed, \autoref{LCHEgroupoid} has already provided the $\downarrow$ part (although strictly speaking should also assume \eqref{ac<bc} holds for arbitrary $s\in S$, not just $s\prec S$, so that the entirety of $(O_s)_{s\in S}$ is an \'etale family).  The $\uparrow$ part and its reversibility comes from the following extension of the Hausdorff case of \cite[Theorem 4.8]{Exel2010} (from bases to pseudobases but, more importantly, even for non-ample groupoids).

\begin{prp}\label{GroupoidRecovery}
If $G$ is Hausdorff with \'etale pseudobasis $P\not\ni\emptyset$ then $(P\cup\{\emptyset\},\Subset)$ is a pseudobasic inverse semigroup.  Moreover, $G$ is isomorphic to $\mathcal{T}(P)$ via
\[g\mapsto T_g=\{O\in P:g\in O\}.\]
\end{prp}

\begin{proof}
It is immediately verified that $P\cup\{\emptyset\}$ is inverse semigroup where $\leq\ =\ \subseteq$ on $P$.  As $\Subset$ is a stronger relation than $\subseteq$, \eqref{LeftAuxiliary} and \eqref{Inverses} for $\prec\ =\ \Subset$ immediately follows.  As $G$ is Hausdorff, \eqref{Multiplicative} holds, by \cite[Proposition 6.4]{BiceStarling2018} (or note products of compact subsets are compact as in the proof of \autoref{CoherentEtale}).  By \autoref{TopologyRecovery}, $P$ is an abstract pseudobasis and $g\mapsto T_g$ is a homeomorphism onto $\mathcal{T}(P)$.  By \eqref{Functor}, $g\mapsto T_g$ is also functor.  Also, if $g^{-1}g\neq hh^{-1}$ then, by \eqref{Separating},
\[(T_g^{-1}T_g)^\leq=T_{g^{-1}g}\neq T_{hh^{-1}}=(T_hT_h^{-1})^\leq,\]
i.e. $(g,h)\in G^{(2)}$ iff $(T_g,T_h)\in\mathcal{T}(P)^{(2)}$.  So $G$ is isomorphic to $\mathcal{T}(P)$ via $g\mapsto T_g$.
\end{proof}

\begin{xpl}
Say $G$ is a discrete group acting on some locally compact Hausdorff space $X$.  Given any pseudobasis $P$ for $X$, we can obtain a larger pseudobasis $P'$ by taking its closure under finite intersections and images of elements of $G$ (and if $X$ is second countable then we can even take $P$ to be a countable (pseudo)basis which will remain countable when enlarged in this way, as long as $G$ is also countable).  Then $S=(g,O)_{g\in G,O\in P'}$ will form a \'etale pseudobasis of the transformation groupoid $(G,X)$.  Applying \autoref{GroupoidRecovery} allows us to recover $(G,X)$ just from the inverse semigroup structure of $S$.
\end{xpl}

\newpage

\bibliography{maths}{}
\bibliographystyle{alphaurl}

\end{document}